\documentclass[reqno,12pt]{amsart}

\usepackage{amsmath, amssymb}
\usepackage{color}

\usepackage{mathrsfs}
\usepackage[hypertex]{hyperref}

\numberwithin{equation}{section}

\newcommand{\1}{\mathbf {1}}
\newcommand{\0}{\mathbf {0}}
\newcommand{\Z}{{\mathbb Z}}
\newcommand{\Q}{{\mathbb Q}}
\newcommand{\C}{{\mathbb C}}

\newcommand{\bsa}{\boldsymbol{a}}

\newcommand{\al}{\alpha}
\newcommand{\be}{\beta}

\newcommand{\dl}{\delta}
\newcommand{\gm}{\gamma}
\newcommand{\lm}{\lambda}
\newcommand{\om}{\omega}

\newcommand{\la}{\langle}
\newcommand{\ra}{\rangle}

\newcommand{\ot}{\otimes}

\newcommand{\mfsl}{\mathfrak{sl}}

\newcommand{\Vaff}{V^{\mathrm{aff}}}

\newcommand{\SC}[1]{\Irr(#1)_{\mathrm{sc}}}
\newcommand{\abs}[1]{\lvert{#1}\rvert}

\DeclareMathOperator{\Aut}{Aut}
\DeclareMathOperator{\ch}{ch}
\DeclareMathOperator{\Com}{Com}

\DeclareMathOperator{\Hom}{Hom}

\DeclareMathOperator{\Irr}{Irr}

\DeclareMathOperator{\spn}{span}
\DeclareMathOperator{\supp}{supp}

\DeclareMathOperator{\wt}{wt}

\newtheorem{theorem}{Theorem}[section]
\newtheorem{proposition}[theorem]{Proposition}
\newtheorem{lemma}[theorem]{Lemma}

\newtheorem{remark}[theorem]{Remark}

\newtheorem{hypothesis}[theorem]{Hypothesis}

\newcommand{\orbit}{\mathscr{O}}
\begin{document}

\title[$\Z_k$-code VOAs]
{$\Z_k$-code vertex operator algebras}

\author[T. Arakawa]{Tomoyuki Arakawa}
\address{Research Institute for Mathematical Sciences,
Kyoto University, Kyoto 606-8502, Japan}
\email{arakawa@kurims.kyoto-u.ac.jp}

\author[H. Yamada]{Hiromichi Yamada}
\address{Department of Mathematics, Hitotsubashi University, Kunitachi,
Tokyo 186-8601, Japan}
\email{yamada.h@r.hit-u.ac.jp}

\author[H. Yamauchi]{Hiroshi Yamauchi}
\address{Department of Mathematics, Tokyo Woman's Christian University, 
Suginami, Tokyo 167-8585, Japan}
\email{yamauchi@lab.twcu.ac.jp}



\begin{abstract}
We introduce a simple, self-dual, rational, and $C_2$-cofinite vertex operator algebra of CFT-type 
associated with a $\Z_k$-code for $k \ge 2$
 based on the $\Z_k$-symmetry among the simple current modules for   
the parafermion vertex operator algebra $K(\mathfrak{sl}_2,k)$. 
We show that it is naturally realized  as the commutant of a certain subalgebra in a lattice vertex operator algebra. 
Furthermore, we construct all the irreducible modules 
inside a module for the lattice vertex operator algebra. 
\end{abstract}

\maketitle

\section{Introduction}\label{sec:intro}

The parafermion vertex operator algebra $K(\mathfrak{g},k)$ associated with 
a finite dimensional simple Lie algebra $\mathfrak{g}$ and a positive integer $k$ 
is by definition the commutant of the Heisenberg vertex operator algebra 
generated by the Cartan subalgebra of $\mathfrak{g}$ in 
$L_{\widehat{\mathfrak{g}}}(k,0)$, where $L_{\widehat{\mathfrak{g}}}(k,0)$ is 
the simple affine vertex operator algebra associated with the affine 
Kac-Moody Lie algebra $\widehat{\mathfrak{g}}$ at level $k$. 
In the case where $\mathfrak{g} = \mathfrak{sl}_2$ and $k \ge 2$, 
$K(\mathfrak{sl}_2,k)$ is  isomorphic to 
a minimal series  principal $W$-algebra of type $A$ 
which is a simple, self-dual, rational, and $C_2$-cofinite vertex operator algebra 
of CFT-type \cite{ALY2019},
and has exactly $k$ simple currents 
$M^j$, $j \in \Z_k$, with $\Z_k$-symmetry. 
That is, those simple currents form a cyclic group of order $k$ with respect to 
the fusion product,  
$M^i \boxtimes_{M^0} M^j = M^{i+j}$ for $i, j \in \Z_k$ with $M^0 = K(\mathfrak{sl}_2,k)$. 

In this article we introduce  a vertex operator algebra $M_D$ 
associated with a $\Z_k$-code $D$ of lenght $\ell$. 
Here, a  $\Z_k$-code $D$ is  an additive subgroup of $(\Z_k)^\ell$. 
For each codeword $\xi = (\xi_1,\ldots,\xi_\ell) \in D$, we associate 
the tensor product $M_\xi = M^{\xi_1} \otimes \cdots \otimes M^{\xi_\ell}$ of 
simple current $K(\mathfrak{sl}_2,k)$-modules $M^{\xi_r}$, $1 \le r \le \ell$. 
Then the direct sum $$M_D = \bigoplus_{\xi \in D} M_\xi$$ has a structure of 
an abelian intertwining algebra \cite[Theorem 4.1]{vEMS2017}.
Furthermore, $M_D$ becomes 
a vertex operator algebra 
if each $M_\xi$ has integral conformal weight \cite[Theorem 4.2]{vEMS2017}. 
Being a $D$-graded simple current exrension of $M_\0=K(\mathfrak{sl}_2,k)^{\otimes \ell}$,
the vertex operator algebra $M_D$ is simple, self-dual, rational, $C_2$-cofinite, 
and of CFT-type with central charge $2\ell (k-1)/(k+2)$
(Theorem \ref{thm:MD}).
Such a construction of $M_D$ was initiated in \cite{Miyamoto1996} for the case $k = 2$, 
and  the properties of the vertex operator algebra $M_D$ for $k = 2$ 
have been studied extensively, 
see \cite{DGH1998, LY2008, Miyamoto1998, Miyamoto2004} and the references therein. 
The vertex operator algebra $M_D$ for $k = 3$ 
was constructed by a slightly different 
method in \cite{KMY2000}, 
and its irreducible modules were studied in \cite{Lam2001b}.

We realize the vertex operator algebra
$M_D$
inside a vertex operator algebra $V_{\Gamma_D}$ associated with 
a certain positive definite even lattice $\Gamma_D$. 
Moreover,
every irreducible $M_D$-module 
is explicitly described
inside a module for the lattice vertex operator algebra $V_{\Gamma_D}$.

More precisely,
consider the
lattice vertex operator algebra 
$V_{\sqrt{2}A_{k-1}}$,
which is an extension of 
the vertex operator algebra
$K(\mathfrak{sl}_2,k)\otimes K(\mathfrak{sl}_k,2)$.
There are cosets $N^{(j)}$, $j \in \Z_k$, of $\sqrt{2}A_{k-1}$ in 
the dual lattice
$(\sqrt{2}A_{k-1})^\circ$ 
such that $N^{(i)} + N^{(j)} = N^{(i+j)}$, and $V_{N^{(j)}}$ contains $M^j$. 
For $\xi = (\xi_1,\ldots,\xi_\ell) \in D$, we consider a coset  
$N(\xi) = N^{(\xi_1)} \times \cdots \times N^{(\xi_\ell)}$ of $(\sqrt{2}A_{k-1})^\ell$ in 
$((\sqrt{2}A_{k-1})^\circ)^\ell$. 
The union $\Gamma_D$ 
of those cosets is a positive definite even lattice if and only if $(\xi | \xi) = 0$ 
for all $\xi \in D$ (Lemma \ref{lem:GD}), 
where $(\,\cdot\, | \,\cdot\,)$ is the standard inner product on $(\Z_k)^\ell$.
Then
 $M_D$ is realized as the commutant of $K(\mathfrak{sl}_k,2)^{\otimes \ell}$ 
in the lattice vertex operator algebra $V_{\Gamma_D}$ (Eq. \eqref{eq:M_D}).

We also consider a necessary and sufficient condition on the code $D$ 
for which $\Gamma_D$ is a positive definite odd lattice, 
and $M_D$ is a vertex operator superalgebra. 

Using the representation theory of simple current extensions 
(Section \ref{subsec:irred_VD-modules}), 
we construct all the irreducible $M_D$-modules inside $V_{(\Gamma_D)^\circ}$, 
where $(\Gamma_D)^\circ$ is the dual lattice of $\Gamma_D$ 
(Theorems \ref{thm:contain_irred}, \ref{thm:stable-module_M0}, 
and \ref{thm:decomp_M_D_X}). 
Any linear character $\chi$ of the finite abelian group $D$ naturally induces 
an automorphism of the vertex operator algebra $M_D$. 
We discuss irreducible $\chi$-twisted $M_D$-modules as well. 
In particular, we obtain the number of inequivalent irreducible $\chi$-twisted 
$M_D$-modules (Theorem \ref{thm:count_irred_twisted_mod}).  
We also study the irreducible $M_D$-modules 
in the case where $M_D$ is a vertex operator superalgebra 
 (Theorem \ref{thm:contain_irred_B}). 

The construction of $M_D$ as a commutant of $K(\mathfrak{sl}_k,2)^{\otimes \ell}$ 
in the lattice vertex operator algebra $V_{\Gamma_D}$ 
was previously discussed in \cite{AYY2016}. 
However, the treatment of the simple current $K(\mathfrak{sl}_2,k)$-modules $M^j$ in $V_{N^{(j)}}$, $j \in \Z_k$,  
was slightly different, and the method there is not suitable for 
all the irreducible $K(\mathfrak{sl}_2,k)$-modules in $V_{(\sqrt{2}A_{k-1})^\circ}$. 
In the present paper, we use decompositions of 
certain irreducible $V_{\sqrt{2}A_{k-1}}$-modules (Proposition \ref{prop:Mij_in_V_Nj0d}), 
from which we know how the irreducible $K(\mathfrak{sl}_2,k)$-modules appear in $V_{(\sqrt{2}A_{k-1})^\circ}$ 
(Proposition \ref{prop:appearance_of_Mij}), 
and it enables us to describe  
the irreducible $M_D$-modules inside $V_{(\Gamma_D)^\circ}$. 

This paper is organized as follows. 
Section \ref{sec:preliminaries} is devoted to preliminaries, where we recall 
the representation theory of simple current extensions. 
In Section \ref{sec:Ksl2k}, we review the properties of 
the parafermion vertex operator algebra $K(\mfsl_2,k)$ for later use. 
In Sections 
\ref{sec:coset-N-in-Ncirc},
\ref{sec:dec-V_Nja}, and 
\ref{sec:irred-mod-in-coset}, 
we describe the cosets of $N=\sqrt{2}A_{k-1}$ in $N^\circ=(\sqrt{2}A_{k-1})^{\circ}$, 
and study how irreducible $K(\mfsl_2,k)$-modules appear in the irreducible 
$V_N$-modules. 
The vertex operator algebra $M_D$  
is defined in Section \ref{sec:gamma_D-M_D}. 
In Section \ref{sec:irred_M_D-modules_A}, 
we study the irreducible twisted and untwisted modules for $M_D$, 
including the classification of irreducible modules, 
and realizations of the irreducible modules in $V_{(N^\circ)^\ell}$. 
In Section \ref{sec:irred_M_D-modules_B}, we discuss the irreducible $M_D$-modules 
in the case where $M_D$ is a vertex operator superalgebra. 
Finally, in Section \ref{sec:examples}, we mention some known examples of $M_D$. 
We calculate the minimal norm of elements in each coset of $N$ in $N^\circ$ 
in Appendix A. 

As to the $P(z)$-tensor product $\boxtimes_{P(z)}$ of \cite{HL1994} 
for a vertex operator algebra $V$, 
we only use it with $z = 1$. 
We write $\boxtimes_V$ for $\boxtimes_{P(1)}$, and call it the fusion product. 
We also use $\otimes$ to denote the tensor product of vertex operator algebras 
and their modules as in \cite{FHL1993}.

\noindent
\textbf{Acknowledgments.}
We would like to thank Ching Hung Lam and Hiroki Shimakura for stimulating discussions and helpful advice.  
The first author was partially supported by JSPS KAKENHI grant No.17H01086 and No.17K18724. 
The third author was partially supported by JSPS KAKENHI grant No.19K03409.

\section{Preliminaries}\label{sec:preliminaries}

In this section, we recall some basic properties of simple current extensions 
of vertex operator algebras and their irreducible modules. 
Our notations for vertex operator algebras and their modules are standard 
\cite{FHL1993, FLM1988, LL2004}.

\subsection{Simple current modules}\label{subsec:simple_current}

Let $V$ be a simple, self-dual, rational, and $C_2$-cofinite 
vertex operator algebra of CFT-type. 
Then a fusion product $M \boxtimes_V N$ over $V$ of any $V$-modules 
$M$ and $N$ exists \cite{HL1995c, Li1998}. 
The fusion product is commutative and associative 
\cite[Theorem 3.7]{Huang2005}. 

We denote by $\Irr(V)$ the set of equivalence classes of irreducible $V$-modules. 
Then 
\begin{equation*}
  M^1 \boxtimes_V M^2 
  = \sum_{M^3 \in \Irr(V)} \dim I_V \binom{M^3}{M^1 \ M^2} M^3
\end{equation*}
for $M^1, M^2 \in \Irr(V)$, where 
$I_V \binom{M^3}{M^1 \ M^2}$ is the set of all intertwining operators of type 
$\binom{M^3}{M^1 \ M^2}$.
An irreducible $V$-module $A$ is called a simple current  
if $A \boxtimes_V X$ is an irreducible $V$-module for any $X \in \Irr(V)$. 
A set $\{ A^\alpha \mid \alpha \in D\}$ of simple current $V$-modules indexed 
by a finite abelian group $D$ is said to be $D$-graded if 
$A^\alpha$, $\alpha \in D$, are inequivalent to each other 
with $A^0 = V$ and 
$A^\alpha \boxtimes_V A^\beta = A^{\alpha + \beta}$, $\alpha, \beta \in D$. 
The set $\SC{V}$ of equivalence classes of simple current $V$-modules is 
graded by a finite abelian group \cite[Corollary 1]{LY2008}. 
The inverse of $A \in \SC{V}$ with respect to the fusion product is 
its contragredient module $A'$.
The fusion product by $A \in \SC{V}$ induces a permutation on $\Irr(V)$. 
\begin{equation}\label{eq:permutation}
  X \mapsto A \boxtimes_V X
\end{equation}

For a $V$-module $X$, we denote its conformal weight by $h(X)$, 
which is a rational number \cite[Theorem 11.3]{DLM2000}. 
We define a map $b_V : \SC{V}\times \Irr(V) \to \Q/\Z$ by
\begin{equation}\label{eq:bilinear_b}
  b_V(A,X) = h(A \boxtimes_V X) - h(A) - h(X) + \Z
\end{equation}
for $A \in \SC{V}$ and $X \in \Irr(V)$. 
The map $b_V$ was introduced in \cite[Section 3]{vEMS2017} 
in the case where $\SC{V}=\Irr(V)$, see also \cite[Section 2]{Moeller2016}. 
A proof of the following lemma can be found in \cite[Section 2]{YY2018}. 

\begin{lemma}\label{lem:bV-bilinear-rev}
  Let $A$, $B\in \SC{V}$, and  $X \in \Irr(V)$.

  \textup{(1)} $b_V(A\boxtimes_V B, X) = b_V(A, X)+b_V(B, X)$.

  \textup{(2)} $b_V(A, B\boxtimes_V X) = b_V(A, B)+b_V(A, X)$.
\end{lemma}

\subsection{Representations of simple current extensions}\label{subsec:irred_VD-modules}

Let $V$ be a simple, self-dual, rational, and $C_2$-cofinite 
vertex operator algebra of CFT-type. 
Let $\{ V^\alpha \mid \alpha \in D \}$ be a $D$-graded set of simple current 
$V$-modules for a finite abelian group $D$ with $V^0 = V$ 
and $h(V^\alpha) \in \frac{1}{2}\Z$ for all $\alpha \in D$. 
Then the direct sum $V_D = \bigoplus_{\alpha \in D}V^\alpha$ 
has a structure of either a simple vertex operator algebra or 
a simple vertex operator superalgebra which extends 
the $V$-module structure on $V_D$ 
\cite[Theorem 3.12]{CKL2015}, see also the references therein. 
Such a simple vertex operator (super)algebra structure on $V_D$ is unique 
\cite[Proposition 5.3]{DM2004}. 
The vertex operator (super)algebra $V_D$ is called 
a $D$-graded simple current extension of $V$. 
In this section, we only consider the case 
in which 
$h(V^\alpha) \in \Z$ for all $\alpha \in D$, and 
$V_D$ is a vertex operator algebra.
It is known that $V_D$ is simple, self-dual, rational, $C_2$-cofinite, 
and of CFT-type \cite[Theorem 2.14]{Yamauchi2004}. 

We recall the representation theory of $V_D$ 
from \cite{Lam2001,Yamauchi2004}.
As to the notion of a $g$-twisted module for a vertex operator algebra 
with respect to its automorphism $g$, 
we adopt the definition in \cite{DLM2000}. 
Thus a $g$-twisted module in \cite{Yamauchi2004} means 
a $g^{-1}$-twisted  module in this paper.

Let $D^* = \Hom(D, \C^\times)$ be the character group of $D$. 
For $\chi \in D^*$, a scalar multiplication by $\chi(\alpha)$ on $V^\alpha$,  
$\alpha \in D$, is an automorphism of the vertex operator algebra $V_D$.  
That is, $D^*$ naturally acts on $V_D$, 
and we can regard $D^*$ as a subgroup of $\Aut V_D$.
Let $M$ be a $\chi$-twisted $V_D$-module for $\chi \in D^*$.
We say $M$ is $D$-graded if there is a decomposition 
$M = \bigoplus_{\alpha\in D} M^\alpha$ as a $V$-module such that 
$0 \ne V^\alpha \cdot M^\beta \subset M^{\alpha+\beta}$ for 
$\alpha$, $\beta\in D$,
where we set 
$V^\alpha \cdot S = \spn\{ a_{(n)} v \mid a \in V^\alpha, v \in S, n \in \Q\}$
for a subset $S$ of $M$.

We consider the action of $D$ on $\Irr(V)$ in \eqref{eq:permutation}. 
Let $\Irr(V) = \bigcup_{i\in I}\orbit_i$ 
be the $D$-orbit decomposition.   
Using the map $b_V$ in \eqref{eq:bilinear_b}, 
we define a map $\chi_X : D\to \C^\times$ by
\begin{equation*}
  \chi_X(\alpha)=\exp(2\pi\sqrt{-1}\, b_V(V^\alpha,X))
\end{equation*}
for $X \in \Irr(V)$. 
The map $\chi_X$ is a linear character of $D$ by (1) of Lemma \ref{lem:bV-bilinear-rev}. 
For a $D$-orbit $\orbit_i$, 
(2) of Lemma \ref{lem:bV-bilinear-rev} implies that $\chi_X$ 
is independent of the choice of $X \in \orbit_i$, 
as $h(V^\alpha) \in \Z$ for all $\alpha \in D$. 
Thus $\chi_X$ is uniquely determined by $\orbit_i$,  
so we can write $\chi_i$ for $\chi_X$. 

We summarize \cite[Theorem 4.4]{Lam2001} and 
\cite[Lemma 2.11, Theorems 2.14, 2.19, 3.2, 3.3]{Yamauchi2004} as follows.

\begin{theorem}\label{thm:induced-module}
Let $V_D$ be a $D$-graded simple current extension of $V$, 
and let $X \in \Irr(V)$. 

\textup{(1)}
There exists a unique structure of a 
$D$-graded $\chi_X$-twisted $V_D$-module on the space 
$V_D\boxtimes_V X = \bigoplus_{\alpha\in D} V^\alpha \boxtimes_V X$ 
which contains $V^0\boxtimes_V X\cong X$ as a $V$-submodule.

\textup{(2)}
If $M = \bigoplus_{\alpha \in D} M^\alpha$ is a $D$-graded $\chi_X$-twisted 
$V_D$-module such that $X \subset M^\alpha$ as a $V$-submodule 
for some $\alpha \in D$, then 
$V_D \cdot X$ is isomorphic to the $D$-graded $\chi_X$-twisted 
$V_D$-module $V_D\boxtimes_V X$ in the assertion \textup{(1)}, where
$V_D \cdot X = \spn\{ a_{(n)} v \mid a \in V_D, v\in X, n \in \Q\}  \subset M$.

\textup{(3)}   
Let $\sigma \in \Aut V_D$ such that $\sigma$ is the identity on $V$. 
Assume that there is a $\sigma$-twisted $V_D$-module containing $X$ 
as a $V$-submodule. 
Then $\sigma = \chi_X$, and 
there exists a surjective $V_D$-homomorphism from $V_D\boxtimes_V X$ onto $V_D \cdot X$.
\end{theorem}

For a $D$-orbit $\orbit_i$ in $\Irr(V)$, 
the structure of a $D$-graded $\chi_X$-twisted $V_D$-module on the space 
$V_D\boxtimes_V X$ 
in (1) of the above theorem is independent of the choice of $X \in \orbit_i$, 
and it is uniquely determined by $\orbit_i$.
The $\chi_X$-twisted $V_D$-module $V_D\boxtimes_V X$ is not necessarily irreducible. 
The assertion (3) of the above theorem implies that $V_D \cdot X$ is isomorphic to 
a direct summand of $V_D\boxtimes_V X$.

Since any irreducible $\chi$-twisted $V_D$-module for $\chi \in D^*$  
is isomorphic to a direct summand 
of the $\chi_X$-twisted $V_D$-module 
$V_D\boxtimes_V X$ with $\chi = \chi_X$ for some $X \in \Irr(V)$ 
by Theorem \ref{thm:induced-module}, 
the study of $\chi$-twisted $V_D$-modules is reduced to the study of 
the $\chi_X$-twisted $V_D$-module $V_D\boxtimes_V X$. 

Let $D_X = \{ \alpha \in D \mid V^\alpha \boxtimes_V X \cong X \}$ 
be the stabilizer of $X \in \Irr(V)$ for the action of $D$ on $\Irr(V)$ in \eqref{eq:permutation}. 
For a $D$-orbit $\orbit_i$, the stabilizer $D_X$ is independent 
of the choice of $X \in \orbit_i$, and it is uniquely determined by $\orbit_i$. 
Hence we can write $D_i$ for $D_X$.

In the case where $D_X = 0$, the following assertion holds 
\cite[Proposition 3.8]{SY2003}.

\begin{proposition}\label{prop:stable-module}
  If $D_X = 0$, then
  $V_D \boxtimes_V X$ is an irreducible $\chi_X$-twisted $V_D$-module.
\end{proposition}

If $D_X$ is non-trivial, 
then the $\chi_X$-twisted $V_D$-module $V_D\boxtimes_V X$ is reducible, 
and we need to take some 2-cocycles of $D_X$ 
into account to obtain its irreducible decomposition as discussed in 
\cite{Lam2001,Yamauchi2004}. 
Let $X \in \Irr(V)$, and assume that $D_X \ne 0$. 
We consider the $D_X$-graded simple current 
extension $V_{D_X} = \bigoplus_{\alpha\in D_X} V^\alpha$ of $V$.
Set $V_{\beta+D_X} = \bigoplus_{\alpha \in \beta+D_X} V^\alpha$ 
for a coset $\beta+D_X \in D/D_X$.
Then $V_D = \bigoplus_{\beta+D_X\in D/D_X} V_{\beta+D_X}$ 
is a $D/D_X$-graded simple current extension of $V_{D_X}$.
Note that 
$V_{D_X} \boxtimes_V X \cong X^{\oplus \abs{D_X}}$ as $V$-modules.
Set $Q = \Hom_V(X, V_{D_X} \boxtimes_V X)$.
Then $\dim Q = \abs{D_X}$, and we have a canonical isomorphism
\begin{equation}\label{eq:mult}
  V_{D_X}\boxtimes_V X \cong X\otimes Q.
\end{equation}

It is shown in \cite[Theorem 3.10]{Lam2001} and 
\cite[Theorems 2.14, 2.19]{Yamauchi2004} that there exists a 
2-cocycle $\epsilon \in Z^2(D_X, \C^\times)$ such that the space $Q$ carries 
a structure of a module for a twisted group algebra $\C^\epsilon[D_X]$ 
associated with $\epsilon$ \cite[Chapter 2]{Karpilovsky1993}.
Indeed, $Q$ is isomorphic to the regular representation of 
$\C^\epsilon[D_X]$. 
If $R$ is a $\C^\epsilon[D_X]$-submodule of $Q$, then the subspace $X\otimes R$ 
of $X\otimes Q$ in \eqref{eq:mult} is a $V_{D_X}$-submodule of $V_{D_X}\boxtimes_V X$. 
Thus the irreducible decomposition of $V_{D_X} \boxtimes_V X$ as a $V_{D_X}$-module 
is obtained by the irreducible decomposition 
of $Q$ as a $\C^\epsilon[D_X]$-module.

Let $T$ be an irreducible $V_{D_X}$-submodule of $V_{D_X} \boxtimes_V X$.
Then $T$ is also a direct sum of some copies of $X$ as a $V$-module,  
and $V_{\beta+D_X} \boxtimes_V T$, $\beta+D_X \in D/D_X$, are inequivalent 
irreducible $V_{D_X}$-modules.
Hence the $\chi_X$-twisted $V_D$-module $V_{D }\boxtimes_{V_{D_X}} T$ is irreducible  
by Proposition \ref{prop:stable-module}. 
The $\chi_X$-twisted $V_D$-module structure of  
$V_{D }\boxtimes_{V_{D_X}} T$ is uniquely determined by $T$. 
Therefore, the irreducible decomposition of $V_D\boxtimes_V X$ as a 
$\chi_X$-twisted $V_D$-module is in one-to-one correspondence with 
the irreducible decomposition of $Q$ in \eqref{eq:mult} 
as a $\C^\epsilon[D_X]$-module.

The determination of the 2-cocycle $\epsilon$ requires more information on 
the associativity constraints of the fusion products 
of $V$-modules \cite{Lam2001,Yamauchi2004}. 
However, we will only deal with the case where $D_X$ 
can be regarded as a binary code in this paper. 
So we make the following assumption.

\begin{hypothesis}\label{hypo:binary}
\textup{(1)} 
$M^0$ is a simple, self-dual, rational, and $C_2$-cofinite vertex operator algebra of CFT-type.

\textup{(2)} 
$M^1$ is a self-dual simple current $M^0$-module such that the 
$\Z_2$-graded simple current extension $M^0\oplus M^1$ of $M^0$ 
is either a simple vertex operator algebra with $h(M^1)\in \Z$ 
or a simple vertex operator superalgebra with $h(M^1)\in \Z + 1/2$.

\textup{(3)}
For any irreducible $M^0$-module $P$, the direct sum $P^0 \oplus P^1$ 
with $P^0=P$ and $P^1=M^1 \boxtimes_{M^0} P$ has a unique 
structure of a $\Z_2$-graded either untwisted or $\Z_2$-twisted 
$M^0 \oplus M^1$-module.

\textup{(4)}
$V=(M^0)^{\otimes n}$ for some $n>0$. 

\textup{(5)}
$X \in \Irr(V)$ with $D_X \ne 0$. 
Moreover, $D_X$ has a structure of a binary code of length $n$, and  
$V^\alpha \cong M^{\alpha_1} \otimes \cdots \otimes M^{\alpha_n}$ 
for $\alpha = (\alpha_1,\dots,\alpha_n)\in D_X$. 
In particular, 
\begin{equation*}
  V_{D_X} = \bigoplus_{\alpha=(\alpha_1,\dots,\alpha_n)\in D_X} 
  M^{\alpha_1} \otimes \cdots \otimes M^{\alpha_n}
  \subset (M^0\oplus M^1)^{\otimes n}
\end{equation*}
as an extension of $V = (M^0)^{\otimes n}$.
\end{hypothesis}

Suppose $V_{D_X}$ satisfies Hypothesis \ref{hypo:binary}.
Under this assumption, we can describe the 2-cocycle 
$\epsilon\in Z^2(D_X,\C^\times)$ explicitly.
We divide our argument into two cases.

\medskip

\paragraph{Case 1}
Suppose $M^0\oplus M^1$ is a simple vertex operator algebra with $h(M^1) \in \Z$. 
By (3) of Hypothesis \ref{hypo:binary}, the 2-cocycle 
$\epsilon \in Z^2(D_X,\C^\times)$ is cohomologous to a 2-coboundary by 
\cite[Chapter 2, Corollary 2.5]{Karpilovsky1993}. 
Hence $Q$ is the regular representation of an ordinary group algebra 
$\C[D_X]$, so that $Q$ is a direct sum of $\abs{D_X}$ 
inequivalent irreducible $\C[D_X]$-modules.
Therefore, $V_{D_X} \boxtimes_V X$ decomposes into a direct sum of 
$\abs{D_X}$ inequivalent irreducible $V_{D_X}$-submodules. 
By considering $V_D$ as a $D/D_X$-graded simple current extension of $V_{D_X}$, 
we see that the irreducible decomposition of $V_D\boxtimes_V X$ 
as a $\chi_X$-twisted $V_D$-module is as follows.

\begin{proposition}\label{prop:decomp1}
Suppose $D_X \ne 0$ and $V_{D_X}$ satisfies Hypothesis \ref{hypo:binary}.
Suppose further that $M^0\oplus M^1$ in (2) of Hypothesis \ref{hypo:binary} 
is a simple vertex operator algebra with $h(M^1) \in \Z$.
Then 
the irreducible decomposition of the $\chi_X$-twisted $V_D$-module 
$V_D\boxtimes_V X$ is given as 
\[
  V_D \boxtimes_V X = \bigoplus_{j = 1}^{\abs{D_X}} U^j,
\]
where $U^j$, $1 \leq j \leq \abs{D_X}$, are inequivalent irreducible 
$\chi_X$-twisted $V_D$-modules. 
Furthermore, $U^j \cong \bigoplus_{W \in \orbit_i} W$ as $V$-modules, 
where $\orbit_i$ is the $D$-orbit in $\Irr(V)$ containing $X$.
\end{proposition}

\paragraph{Case 2}
Suppose $M^0\oplus M^1$ is a simple vertex operator 
superalgebra with $h(M^1)\in \Z+1/2$. 
In this case, $D_X$ is an even binary code, as the conformal weight of 
$V^\alpha \cong M^{\alpha_1} \otimes \cdots \otimes M^{\alpha_n}$ 
is an integer for $\alpha = (\alpha_1,\dots,\alpha_n) \in D_X$.
By (3) of Hypothesis \ref{hypo:binary}, 
we can find the 2-cocycle $\epsilon$ inside 
$Z^2(D_X,\{\pm 1\})$ 
which satisfies
\begin{equation}\label{eq:2-cocycle}
  \epsilon(\alpha,\alpha) = (-1)^{\wt(\alpha)/2}, \quad
  \epsilon(\alpha,\beta) \epsilon(\beta,\alpha) = (-1)^{(\alpha|\beta)}
\end{equation}
for $\alpha$, $\beta \in D_X$, 
where $\wt(\alpha)$ is the Hamming weight of $\alpha$,  
and $(\, \cdot\, | \, \cdot \,)$ is the standard inner product on $(\Z_2)^n$ 
\cite[Section 4.1]{LY2008}, see also \cite{Miyamoto1996, Miyamoto1998}.
The conditions above uniquely determine the class of $\epsilon$ in 
$H^2(D_X,\{\pm 1\})$ \cite[Proposition 5.3.3]{FLM1988}.

It is shown in \cite[Theorem 5.5.1]{FLM1988} that each irreducible representation
of $\C^\epsilon[D_X]$ is induced from an irreducible representation of 
its maximal commutative subalgebra, and the equivalence classes 
of irreducible $\C^\epsilon[D_X]$-modules are distinguished by their central characters.
Let $D_X^\perp = \{ \alpha \in (\Z_2)^n \mid (\alpha | D_X) = 0\}$ be the 
dual code of the binary code $D_X$, 
and let $E$ be a maximal self-orthogonal subcode of $D_X$.
It follows from \eqref{eq:2-cocycle} that the center of $\C^\epsilon[D_X]$ 
is $\C^\epsilon[D_X \cap D_X^\perp]$, 
and $\C^\epsilon[E]$ is a 
maximal commutative subalgebra of $\C^\epsilon[D_X]$.
Since $\C^\epsilon[D_X \cap D_X^\perp] \cong \C[D_X \cap D_X^\perp]$ 
is an ordinary group algebra, 
the number of inequivalent irreducible representations of 
$\C^\epsilon[D_X]$ 
is equal to that of $\C[D_X \cap D_X^\perp]$, which coincides with 
the order $\abs{D_X \cap D_X^\perp}$ of $D_X \cap D_X^\perp$.
Each irreducible $\C^\epsilon[D_X]$-module has dimension 
$[D_X:E]=[E:D_X \cap D_X^\perp]$, 
namely, $[D_X : D_X \cap D_X^\perp]^{1/2}$ \cite[Theorem 5.5.1]{FLM1988}.
Since the space $Q$ in \eqref{eq:mult} is isomorphic to 
the regular representation of $\C^\epsilon[D_X]$, 
the irreducible decomposition of $V_D\boxtimes_V X$ 
as a $\chi_X$-twisted $V_D$-module is as follows.

\begin{proposition}\label{prop:decomp2}
Suppose $D_X \ne 0$ and $V_{D_X}$ satisfies Hypothesis \ref{hypo:binary}.
Suppose further that $M^0 \oplus M^1$ in (2) of Hypothesis \ref{hypo:binary} 
is a simple vertex operator superalgebra with $h(M^1)\in \Z+1/2$. 
Then 
the irreducible decomposition of the $\chi_X$-twisted $V_D$-module 
$V_D\boxtimes_V X$ is given as 
\[
  V_D\boxtimes_V X 
  = \bigoplus_{j=1}^{\abs{D_X \cap D_X^\perp}} (U^j)^{\oplus m}, 
\]
where $m = [D_X : D_X \cap D_X^\perp]^{1/2}$, and 
$U^j$, $1 \leq j \leq \abs{D_X \cap D_X^\perp}$, are inequivalent irreducible 
$\chi_X$-twisted $V_D$-modules. 
Furthermore, $U^j \cong \bigoplus_{W \in \orbit_i} W^{\oplus m}$ 
as $V$-modules, where $\orbit_i$ is the $D$-orbit in $\Irr(V)$ containing $X$.
\end{proposition}

\section{Parafermion vertex operator algebra $K(\mfsl_2,k)$}
\label{sec:Ksl2k}

In this section, we recall the properties of the parafermion vertex operator algebra $K(\mfsl_2,k)$ for $2 \le k \in \Z$. 
If $k=2$, then $K(\mfsl_2,2)$ is isomorphic 
to the Virasoro vertex operator algebra $L(1/2,0)$ of central charge $1/2$. 
So we assume that $k \ge 3$ for the rest of this section. 

Let $\{ h, e, f\}$ be a standard Chevalley basis of the Lie algebra $\mfsl_2$.   
Let $L_{\widehat{\mfsl}_2}(k,0)$ be the simple affine vertex operator algebra 
associated with $\widehat{\mfsl}_2$ and level $k$.
Then $K(\mfsl_2,k)$ is defined to be the commutant 
of the Heisenberg vertex operator algebra generated by $h(-1)\1$ 
in $L_{\widehat{\mfsl}_2}(k,0)$ \cite{DLWY2010, DLY2009, DL1993}.

We follow the notaions in \cite[Section 4]{DLY2009}. 
Let $L = \Z\al_1 + \cdots + \Z\al_k$ with
$\la \al_i,\al_j \ra = 2\dl_{i,j}$ and $\gm = \al_1 + \cdots + \al_k$. 
Let $H$, $E$, and $F \in V_L$ be as in \cite[Section 4]{DLY2009}. 
Then the component operators $H_{(n)}$, $E_{(n)}$, $F_{(n)}$, $n \in \Z$, 
give a level $k$ representation of $\widehat{\mfsl}_2$ under the
correspondence 
$h(n) \leftrightarrow H_{(n)}$, $e(n) \leftrightarrow E_{(n)}$, $f(n) \leftrightarrow F_{(n)}$,
and the subalgebra $\Vaff$ of the vertex operator algebra 
$V_L \cong L_{\widehat{\mfsl}_2}(1,0)^{\otimes k}$ 
generated by $H$, $E$, and $F$ is isomorphic to $L_{\widehat{\mfsl}_2}(k,0)$. 
We identify $\Vaff$ with $L_{\widehat{\mfsl}_2}(k,0)$.  
We also identify $H_{(n)}$, $E_{(n)}$, and $F_{(n)}$ with $h(n)$, $e(n)$, and $f(n)$, 
respectively.
Let
\begin{equation*}
  M^j = \{ v \in L_{\widehat{\mfsl}_2}(k,0) \mid H_{(n)}v = -2j \dl_{n,0} v 
\text{ for } n \ge 0 \}.
\end{equation*}
Then $M^0=K(\mfsl_2,k)$, and 
$L_{\widehat{\mfsl}_2}(k,0) = \bigoplus_{j=0}^{k-1} M^j \ot V_{\Z\gm - j\gm/k}$ 
as $M^0 \otimes V_{\Z\gm}$-modules \cite[Lemma 4.2]{DLY2009}. 
The index $j$ of $M^j$ can be considered to be modulo $k$. 

Let $L^\circ = \frac{1}{2}L$ be the dual lattice of $L$, and  
let $v^i$, $0 \le i \le k$, and $v^{i,j}$, $ 0 \le j \le i$, be as in \cite[Section 4]{DLY2009}. 
Then the $\Vaff$-submodule $\Vaff\cdot v^i$ of $V_{L^\circ}$ generated by $v^i$ is
isomorphic to an irreducible $ L_{\widehat{\mfsl}_2}(k,0)$-module 
$L_{\widehat{\mfsl}_2}(k,i)$ with top level $\spn\{ v^{i,j} \mid 0 \le j \le i \}$ 
of conformal weight $i(i+2)/4(k+2)$ \cite{FZ1992}, \cite[Section 6.2]{LL2004}. 
Let
\begin{equation*}
M^{i,j} = \{ v \in \Vaff \cdot v^i \mid H_{(n)} v = (i-2j)\delta_{n,0}v \text{ for } n \ge 0\}
\end{equation*}
for $0 \le i \le k$, $0 \le j \le k - 1$. 
Then 
\begin{equation}\label{eq:dec-LUi}
  L_{\widehat{\mfsl}_2}(k,i) = \bigoplus_{j=0}^{k-1} M^{i,j} \otimes V_{\Z\gm + (i-2j)\gm/2k}
\end{equation}
as $M^0 \otimes V_{\Z\gm}$-modules \cite[Lemma 4.3]{DLY2009}. 
The index $j$ of $M^{i,j}$ can be considered to be modulo $k$.

The $-1$ isometry of the lattice $L$ lifts to an automorphism $\theta$ of 
the vertex operator algebra $V_L$ of order $2$. 
Actually, $\theta(H) = -H$, $\theta(E) = F$, and $\theta(F) = E$. 

We summarize the properties of $M^0 = K(\mathfrak{sl}_2,k)$  
\cite{ALY2014, ALY2019, DLWY2010, DLY2009, DW2016}.

(1) $M^0$ is a simple, self-dual, rational, and $C_2$-cofinite 
vertex operator algebra of CFT-type with central charge $2(k-1)/(k+2)$. 

(2) $\ch M^0 = 1 + q^2 + 2q^3 + \cdots$.

(3) $M^0$ is generated by its conformal vector $\omega$ and 
a primary vector $W^3$ of weight $3$.

(4) The automorphism group $\Aut M^0$ of $M^0$ is generated by $\theta$, 
and $\theta(W^3) = -W^3$.

(5) The irreducible $M^0$-modules $M^{i,j}$'s are not always inequivalent. In fact, 
\begin{equation}\label{eq:equiv-Mij}
M^{i,j} \cong M^{k-i,j-i}, \quad 0 \le i \le k, 0 \le j \le k-1.
\end{equation}

(6) $M^{i,j}$, $0 \le j < i \le k$, form a complete set of representatives of 
the equivalence classes of irreducible $M^0$-modules. 

(7) The top level of $M^{i,j}$ is a one dimensional space $\C v^{i,j}$, and its weight is 
\begin{equation}\label{eq:top-weight-Mij}
h(M^{i,j}) = \frac{1}{2k(k+2)}\Big( k(i-2j) - (i-2j)^2 + 2k(i-j+1)j \Big)
\end{equation}
for $0 \le j \le i \le k$. 
Note that \eqref{eq:top-weight-Mij} is valid even when $j = i$.
Any irreducible $M^0$-module except for $M^0$ itself has positive conformal weight.

(8) The automorphism $\theta$ of $M^0$ induces a permutation 
$M^{i,j} \mapsto M^{i,j} \circ \theta \cong M^{i,i-j}$ 
on the irreducible $M^0$-modules for $0 \le i \le k$, $0 \le j \le k-1$.

(9) $M^j$, $0 \le j \le k-1$, are the simple currents with $h(M^j) = j(k-j)/k$, and
\begin{equation}\label{eq:fusion_rule_for_M0}
M^{j'} \boxtimes_{M^0} M^{i,j} = M^{i,j+j'}, \quad 0 \le i \le k, 0 \le j, j' \le k-1.
\end{equation}

The following lemma is a consequence of \eqref{eq:equiv-Mij} and \eqref{eq:fusion_rule_for_M0}.

\begin{lemma}\label{eq:exceptional-id}
$M^{j'} \boxtimes_{M^0} M^{i,j} \cong M^{i,j}$ 
if and only if $j' = 0$, or $k$ is even and $j' = i = k/2$.
\end{lemma}

Let
\[
  N = \{\al \in L \mid \la \al,\gm \ra = 0 \}. 
\]
Then 
$M^0 = \Com_{\Vaff}(V_{\Z\gamma}) \subset \Com_{V_L}(V_{\Z\gamma}) = V_N$.
The commutant of $\Vaff$ in $V_L$ is isomorphic to 
the parafermion vertex operator algebra $K(\mfsl_k,2)$ \cite{Lam2014}. 
We denote it by $T$. 
Thus $T = \Com_{V_L}(\Vaff) = \Com_{V_N}(M^0) \cong K(\mfsl_k,2)$.

\section{Cosets $N(j,\bsa)$ of $N$ in $N^\circ$}
\label{sec:coset-N-in-Ncirc}

We keep the notations in Section \ref{sec:Ksl2k}. 
In this section, we describe the cosets of $N$ in its dual lattice $N^\circ$.
For $\bsa = (a_1,\ldots,a_k) \in \{ 0,1\}^k$, set 
$\delta_{\bsa} = \frac{1}{2} \sum_{p=1}^k a_p \al_p$. 
Then $L^\circ = \bigcup_{\bsa \in \{ 0,1\}^k} (L + \delta_{\bsa})$ 
is the coset decomposition of $L^\circ$ by $L$. 
Let $\beta_p = \al_p - \al_{p+1}$, $1 \le p \le k-1$,  
so $\{\beta_1, \ldots, \beta_{k-1}\}$ is a $\Z$-basis of $N$. 
Set $R = N \oplus \Z\gamma$. 
Then $R \subset L \subset L^\circ \subset R^\circ$ 
with $R^\circ = N^\circ \oplus (\Z\gm)^\circ$ and 
$(\Z\gm)^\circ = \Z\frac{1}{2k}\gm$.
Let 
\begin{equation*}
\lm_k
= \frac{1}{2k} (\beta_1 + 2\beta_2+ \cdots + (k-1)\beta_{k-1})\\
= \frac{1}{2k}\gm - \frac{1}{2}\al_k.
\end{equation*}
Then $\la \beta_p, \lm_k \ra = \delta_{p,k-1}$, $1 \le p \le k-1$, and 
$\la \lm_k, \lm_k \ra = \frac{1}{2} - \frac{1}{2k}$. 
The following lemma holds.

\begin{lemma}\label{lem:dual_N}
\textup{(1)} 
$\{ \beta_2/2, \ldots, \beta_{k-1}/2, \lm_k \}$
is a $\Z$-basis of $N^\circ$.

\textup{(2)} 
The coset decomposition of $N^\circ$ by $N$ is given as
\begin{equation*}
  N^\circ = \bigcup_{\substack{0 \le i \le 2k-1\\d_2, \ldots, d_{k-1} \in \{0,1\}}}
  ( N + d_2\beta_2/2 + \cdots + d_{k-1}\beta_{k-1}/2 + i \lm_k ).
\end{equation*}

\textup{(3)} $N^\circ/N \cong \Z_2^{k-2} \times \Z_{2k}$.
\end{lemma}

We consider another $\Z$-basis of $N^\circ$. 
Let
\begin{equation*}
  \lambda_p 
  = \lambda_k - \frac{1}{2}\beta_p - \cdots - \frac{1}{2}\beta_{k-1} 
  = \frac{1}{2k} \gamma - \frac{1}{2} \alpha_p,
  \quad 1 \le p \le k-1.
\end{equation*}
Then $\lambda_p \in N^\circ$  
and $2\lambda_p  \equiv  2\lambda_k \pmod{N}$. 
Note that
\begin{equation}\label{eq:relation_lambda_p_1}
  \lambda_1 + \cdots + \lambda_k = 0.
\end{equation}

Lemma \ref{lem:dual_N} implies the next lemma.

\begin{lemma}\label{lem:dual_N_2}
\textup{(1)} 
$\{ \lm_2, \ldots, \lm_{k-1}, \lm_k \}$ is a $\Z$-basis of $N^\circ$. 

\textup{(2)} 
The coset decomposition of $N^\circ$ by $N$ is given as
\begin{equation*}
N^\circ = \bigcup_{\substack{0 \le i \le 2k-1\\d_2, \ldots, d_{k-1} \in \{0,1\}}}
\big( N + d_2 \lm_2 + \cdots + d_{k-1} \lm_{k-1}  + i \lm_k \big).
\end{equation*}
\end{lemma}

The coset decomposition of $L$ by $R$ is given as
\begin{equation}\label{eq:dec_L_R}
  L = \bigcup_{j=0}^{k-1} \big( R - j \al_k \big) 
  = \bigcup_{j=0}^{k-1} \big( R + 2 j\lm_k - \frac{j}{k}\gm \big),
\end{equation}
and $L/R \cong \Z_k$.
Moreover, the coset decomposition of $R^\circ$ by $L^\circ$ is given as
\begin{equation*}
  R^\circ = \bigcup_{j=0}^{k-1} \big( L^\circ - \frac{j}{2k}\gm \big),
\end{equation*}
and $R^\circ/L^\circ \cong \Z_k$. 

For $\bsa = (a_1, \ldots, a_k) \in \{0,1\}^k$, the support $\supp (\bsa)$ is the set of 
$p$, $1 \le p \le k$, for which $a_p \ne 0$, and the Hamming weight $\wt (\bsa)$ 
is the number of nonzero entries $a_p$. 
Then
\begin{equation*}
  \delta_{\bsa} 
  = -\sum_{p = 1}^k a_p \lm_p + \frac{\wt (\bsa)}{2k} \gm.
\end{equation*}

For $\bsa = (a_1, \ldots, a_k) \in \{0,1\}^k$, let 
\begin{equation}\label{def:N_j_a}
  N(j,\bsa) = N - \sum_{p = 1}^k a_p \lm_p + 2 j \lm_k, \quad 0 \le j \le k-1.
\end{equation}
Since $2k \lambda_k \in N$, we can consider $j$ to be modulo $k$. 
We have
\begin{equation*}
  N(j, \bsa) + N(j', \bsa') 
  = N(j + j' - (\wt(\bsa) + \wt(\bsa') - \wt(\bsa + \bsa'))/2, \bsa + \bsa'),
\end{equation*}
where $\bsa + \bsa'$ is the sum of $\bsa$ and $\bsa'$ as elements of $(\Z_2)^k$, 
that is, the symmetric difference as subsets of $\{0,1\}^k$.
By the definition of $\lambda_p$, we also have
\begin{equation}\label{eq:N_j_a-bis}
N(j,\bsa) = N + \frac{1}{2} \sum_{p=1}^k a_p \al_p - j \al_k + \frac{2j -  \wt(\bsa)}{2k} \gamma.
\end{equation}
Since 
$2 \lambda_k - \gamma/k = - \al_k$, 
this equation implies that
\begin{equation*}
R + \delta_{\bsa} + 2 j \lm_k - \frac{j}{k} \gm 
= N(j, \bsa) + \Big( \Z\gm + \frac{\wt(\bsa) - 2j}{2k} \gm \Big)
\end{equation*}
as subsets of $R^\circ = N^\circ \oplus (\Z\gm)^\circ$. 
Hence it follows from \eqref{eq:dec_L_R} that
\begin{equation}\label{eq:dec_L_coset_in_dual_R}
L + \delta_{\bsa} = \bigcup_{j=0}^{k-1} 
\Big( N(j,\bsa) + \Big( \Z\gm + \frac{\wt(\bsa) - 2j}{2k} \gm \Big) \Big).
\end{equation}

\begin{lemma}\label{lem:dual_N_3}
\textup{(1)} For $ 0 \le j, j' \le k-1$ and 
$\bsa, \bsa' \in \{0,1\}^k$, 
we have 
$N(j,\bsa) = N(j',\bsa')$ if and only if one of the following conditions holds.

\quad \textup{(i)} $j \equiv j' \pmod{k}$ and $\bsa = \bsa'$.

\quad \textup{(ii)} $j' \equiv j - \wt (\bsa) \pmod{k}$ and $\bsa + \bsa'= (1, \ldots, 1)$.

\textup{(2)} $N(j,\bsa)$, $0 \le j \le k-1$, $\bsa \in \{0,1\}^k$ with $j < \wt(\bsa)$, are 
the distinct cosets of $N$ in $N^\circ$. 
\end{lemma}

\begin{proof}
Clearly, $N(j,\bsa) = N(j',\bsa')$ if the condition (i) holds. 
Suppose the condition (ii) holds. 
Then $N(j,\bsa) = N(j',\bsa')$ 
by \eqref{eq:relation_lambda_p_1} and \eqref{def:N_j_a}. 
Set $i = \wt (\bsa)$ and $i' = \wt (\bsa')$, 
and assume that $j < i$. 
Then $0 \le j < i \le k$ and $0 \le i' \le j' < k$. 
The number of pairs $(j, \bsa)$ with $0 \le j \le k-1$ and $\bsa \in \{0,1\}^{k}$ is $2^{k}k$. 
Since $|N^\circ/N| = 2^{k-1}k$, we see that $N(j,\bsa) = N(j',\bsa')$ only if 
$j$, $j'$, $\bsa$, and $\bsa'$ satisfy the conditions (i) or (ii). 
Hence the assertions (1) and (2) hold.
\end{proof}

\begin{remark}
In Case \textup{(ii)} of Lemma \ref{lem:dual_N_3} \textup{(1)}, we have 
$(\wt(\bsa') - 2j') - (\wt(\bsa) - 2j) \equiv k \pmod{2k}$. 
This agrees with the fact that 
$N(j,\bsa) + ( \Z\gm + \frac{\wt(\bsa) - 2j}{2k} \gm )$, 
$0 \le j \le k-1$, $\bsa \in \{ 0,1 \}^{k}$, 
in \eqref{eq:dec_L_coset_in_dual_R} are the distinct cosets of $R$ in $L^\circ$.
\end{remark}

The next lemma also holds.

\begin{lemma}\label{lem:theta_on_coset_of_N}
The $-1$ isometry $N^\circ \to N^\circ$; $\al \mapsto -\al$ transforms  
$N(j, \bsa)$ into $N(\wt(\bsa)-j, \bsa)$. 
\end{lemma}

\section{Decomposition of $V_{N(j,\bsa)}$}\label{sec:dec-V_Nja}

We keep the notations in Sections \ref{sec:Ksl2k} and 
\ref{sec:coset-N-in-Ncirc}. 
In this section, we study a decomposition of the irreducible $V_N$-module $V_{N(j,\bsa)}$ 
as a direct sum of irreducible modules for a tensor product of $k-1$ Virasoro vertex operator 
algebras and $M^0$. 
Let
\begin{equation*}
c_m = 1 - \frac{6}{(m+2)(m+3)}
\end{equation*}
for $m = 1,2, \ldots,$ and let
\begin{equation*}
h^m_{r,s} = \frac{\big( r(m+3) - s(m+2) \big)^2 - 1}{4(m+2)(m+3)}
\end{equation*}
for $1 \le r \le m+1$, $1 \le s \le m+2$. 
Then $h^m_{r,s} = h^m_{m+2-r, m+3-s}$, and
$L(c_m, h^m_{r, s})$, $1 \le s \le r \le m+1$, 
form a complete set of representatives of the equivalence classes of irreducible modules 
for the Virasoro vertex operator algebra $L(c_m, 0)$ \cite{Wang1993}. 
We denote the conformal vector of $L(c_m, 0)$ by $\omega^m$. 

Recall that 
$\omega$ is the conformal vector of $M^0$.  
Let $\omega_T$ be the conformal vector of $T = \Com_{V_N}(M^0)$. 
Then the conformal vector $\omega_N = \omega_T + \omega$ of $V_N$ is a sum 
of mutually orthogonal Virasoro vectors $\omega^1, \ldots, \omega^{k-1}$, and $\omega$ 
\cite{DLMN1998, LY2004} 
with $\omega_T = \omega^1 + \cdots + \omega^{k-1}$. 
The vector $\omega^m$ generates $L(c_m,0)$, 
so $T \supset L(c_1, 0) \otimes \cdots \otimes L(c_{k-1},0)$. 
The following decomposition is known 
\cite{KR2013, LLY2003, Wakimoto2001}.

\begin{lemma}\label{lem:dec_V_L_coset_LY}
For $\bsa = (a_1,\ldots,a_{k}) \in \{ 0,1\}^{k}$, 
\begin{equation*}
V_{L + \delta_{\bsa}} = 
\bigoplus_{\substack{0 \le i_s \le s\\i_s \equiv b_s \pmod{2}\\1 \le s \le k}}
L(c_1, h^1_{i_1+1, i_2+1}) \otimes \cdots \otimes L(c_{k-1},h^{k-1}_{i_{k-1}+1, i_{k}+1}) 
\otimes L_{\widehat{\mfsl}_2}(k,i_{k})
\end{equation*}
as $L(c_1,0) \otimes \cdots \otimes L(c_{k-1}, 0) \otimes L_{\widehat{\mfsl}_2}(k,0)$-modules, where $b_s = \sum_{p=1}^s a_p$. 
\end{lemma}

Combining the decomposition \eqref{eq:dec-LUi} with Lemma \ref{lem:dec_V_L_coset_LY},
we have
\begin{multline}\label{eq:dec_V_L_coset_1}
V_{L + \delta_{\bsa}} =
\bigoplus_{j = 0}^{k-1} \Big(
\bigoplus_{\substack{
0 \le i_s \le s\\
i_s \equiv b_s \pmod{2}\\
1 \le s \le k}}
L(c_1, h^1_{i_1+1, i_2+1}) \otimes \cdots \otimes L(c_{k-1},h^{k-1}_{i_{k-1}+1, i_{k}+1})\\
\otimes M^{i_{k},j} \otimes V_{\Z\gm + (i_{k}-2 j)\gm/2k} \Big)
\end{multline}
as $L(c_1,0) \otimes \cdots \otimes L(c_{k-1},0) \otimes M^0 \otimes V_{\Z\gm}$-modules. 

Since $b_{k} = \wt (\bsa)$, \eqref{eq:dec_L_coset_in_dual_R} implies that 
\begin{equation}\label{eq:dec_V_L_coset_2}
V_{L + \delta_{\bsa}} =
\bigoplus_{j=0}^{k-1} V_{N(j,\bsa)} \otimes V_{\Z\gm +(b_{k} - 2j)\gm/2k}
\end{equation}
as $V_N \otimes V_{\Z\gm}$-modules.

As $V_{\Z\gm}$-modules, 
$V_{\Z\gm +(b_{k} - 2 j)\gm/2k} \cong  V_{\Z\gm + (i_{k}-2 q)\gm/2k}$ 
if and only if $q \equiv j + (i_{k} - b_{k})/2 \pmod{k}$. 
Here, note that $i_{k}$ on the right hand side of \eqref{eq:dec_V_L_coset_1} 
satisfies $i_{k} \equiv b_{k} \pmod{2}$. 
Comparing \eqref{eq:dec_V_L_coset_1} and \eqref{eq:dec_V_L_coset_2}, we have 
the following theorem, see \cite[Proposition 3.4]{LS2008}. 

\begin{theorem}\label{thm:dec_V_Nja}
For $0 \le j \le k-1$ and $\bsa = (a_1,\ldots,a_{k}) \in \{0,1\}^{k}$, 
the irreducible $V_N$-module $V_{N(j,\bsa)}$ decomposes as a direct sum
\begin{equation}\label{eq:dec_V_Nja}
V_{N(j,\bsa)} = 
\bigoplus_{\substack{
0 \le i_s \le s\\
i_s \equiv b_s \pmod{2}\\
1 \le s \le k}}
L(c_1, h^1_{i_1+1, i_2+1}) \otimes \cdots \otimes L(c_{k-1},h^{k-1}_{i_{k-1}+1, i_{k}+1})
\otimes M^{i_{k}, j + (i_{k} - b_{k})/2}
\end{equation}
of irreducible $L(c_1,0) \otimes \cdots \otimes L(c_{k-1},0) \otimes M^0$-modules, 
where $b_s = \sum_{p=1}^s a_p$.
\end{theorem}

The next remark is a restatement of \cite[Proposition 3.5]{LS2008}.
\begin{remark}
$N(j,\bsa) = N(j',\bsa')$ for $j'$, $\bsa'$ 
in Case \textup{(ii)} of Lemma \ref{lem:dual_N_3} \textup{(1)} corresponds to the following properties of 
the highest weights $h^m_{p,q}$ for 
$L(c_m,0)$ 
and the irreducible modules $M^{i,j}$ for 
$K(sl_2, k)$.

\textup{(1)} $h^m_{p,q} = h^m_{m+2-p, m+3-q}$ 
for $1 \le p \le m+1$, $1 \le q \le m+2$.

\textup{(2)} $M^{i,j} \cong M^{k-i, j-i}$ as $K(sl_2, k)$-modules 
for $0 \le i \le k$, $j \in \Z_{k}$.
\end{remark}

We note that for a given $\bsa \in \{ 0, 1\}^{k}$, 
the $L(c_1,0) \otimes \cdots \otimes L(c_{k-1},0)$-modules 
\begin{equation*}
L(c_1, h^1_{i_1+1, i_2+1}) \otimes \cdots \otimes L(c_{k-1}, h^{k-1}_{i_{k-1}+1, i_{k}+1}), 
\end{equation*}
$0 \le i_s \le s$, $i_s \equiv b_s \pmod{2}$, $1 \le s \le k$, 
in \eqref{eq:dec_V_Nja} are inequivalent to each other.

\section{Irreducible $K(\mfsl_2,k)$-modules in $V_{N(j,\bsa)}$}
\label{sec:irred-mod-in-coset}

In this section, we discuss how irreducible $K(\mfsl_2,k)$-modules $M^{i,j}$ 
appear on the right hand side of \eqref{eq:dec_V_Nja}. 
Since $h^s_{p,q} = 0$ if and only if $(p,q) = (1,1)$ or $(s+1,s+2)$, 
the following lemma holds.

\begin{lemma}\label{lem:L_cm_0}
Let $1 \le m < k$. 
Then for $a_1, \ldots, a_{m+1} \in \{ 0,1\}$ and $0 \le i_s \le s$, $1 \le s \le m+1$, 
the two conditions 
$i_s \equiv b_s \pmod{2}$, $1 \le s \le m+1$, and 
$h^s_{i_s + 1, i_{s+1} + 1} = 0$, $1 \le s \le m$, 
hold only if 
\textup{(i)} $a_s = 0$ and $i_s = 0$, $1 \le s \le m+1$, or 
\textup{(ii)} $a_s = 1$ and $i_s = s$, $1 \le s \le m+1$.
\end{lemma}

For an arbitrarily given $a_1  \in \{0,1\}$, 
each coset of $N$ in $N^\circ$ is uniquely expressed as $N(j,\bsa)$, $j \in \Z_{k}$, 
$\bsa = (a_1,a_2,\ldots,a_{k})$, $a_2, \ldots, a_{k} \in \{ 0,1\}$ by Lemma \ref{lem:dual_N_3}. 
For the rest of this section, we take $a_1 = 0$. 
For simplicity of notation, we omit $\1 \otimes \cdots \otimes \1$ 
in an equation as 
\begin{equation*}
  \{ v \in V_N \mid \omega^s_{(1)} v = 0, 1 \le s \le k-1 \} 
  = \1 \otimes \cdots \otimes \1 \otimes M^0.
\end{equation*}

The following two propositions are clear from Theorem \ref{thm:dec_V_Nja} and Lemma \ref{lem:L_cm_0}.
\begin{proposition}\label{prop:Mj_in_V_Nj0}
For $j \in \Z_{k}$, we have
\begin{equation*}
\{ v \in V_{N(j,(0,\ldots,0))} \mid \omega^s_{(1)} v = 0, 1 \le s \le k-1 \} = M^j.
\end{equation*}
\end{proposition}

\begin{proposition}\label{prop:Mij_in_V_Nj0d}
For $j \in \Z_{k}$ and $d \in \{ 0,1 \}$, we have
\begin{multline}\label{eq:Mij_in_V_Nj0d}
\{ v \in V_{N(j,(0,\ldots,0,d))} \mid \omega^s_{(1)} v = 0, 1 \le s \le k-2 \}\\
= \bigoplus_{\substack{
0 \le i \le k\\
i \equiv d \pmod{2}}}
L(c_{k-1},h^{k-1}_{1, i+1}) 
\otimes M^{i, j + (i - d)/2}. 
\end{multline}
\end{proposition}

The next proposition is a consequence of \eqref{eq:equiv-Mij}.

\begin{proposition}\label{prop:appearance_of_Mij}
Let $d \in \{ 0, 1 \}$.

\textup{(1)} If $k$ is odd, then 
$M^{i, j + (i - d)/2}$, $j \in \Z_{k}$, $0 \le i \le k$, $i \equiv d \pmod{2}$,  
are inequivalent to each other, and 
they are the $k(k+1)/2$ inequivalent irreducible modules $M^{i,j}$, $0 \le j < i \le k$.

\textup{(2)} If $k$ is even, then
$M^{i, j + (i - d)/2}$, $j \in \Z_{k}$, $0 \le i \le k$, $i \equiv d \pmod{2}$,  
cover twice the set of  inequivalent irreducible modules 
$M^{i,j}$, $0 \le j < i \le k$ with $i \equiv d \pmod{2}$. 
There are $k(k+2)/4$ \textup{(}resp.  $k^2/4$\textup{)} inequivalent irreducible modules 
$M^{i,j}$, $0 \le j < i \le k$ with $i \equiv 0 \pmod{2}$ 
\textup{(}resp. $i \equiv 1 \pmod{2}$\textup{)}.
Moreover, for a fixed $j \in \Z_{k}$, the irreducible modules 
$M^{i, j + (i-d)/2}$, $0 \le i \le k$, $i \equiv d \pmod{2}$,  
are inequivalent to each other.
\end{proposition}

\section{$\Gamma_D$ and $M_D$ for a $\Z_k$-code $D$}
\label{sec:gamma_D-M_D}

In this section, we define a vertex operator algebra or a vertex operator superalgebra 
$M_D$ for a $\Z_k$-code $D$. 
The arguments are essentially the same as in Section 3 of \cite{AYY2016}. 

Let $\ell$ be a fixed positive integer. 
A $\Z_k$-code of length $\ell$ means an additive subgroup of $(\Z_k)^\ell$. 
We denote by $(\, \cdot\, | \, \cdot \,)$ the standard inner product 
$(\xi | \eta) = \xi_1 \eta_1 + \cdots + \xi_\ell \eta_\ell \in \Z_k$ 
for $\xi = (\xi_1, \ldots, \xi_\ell),\, \eta = (\eta_1,\ldots, \eta_\ell) \in (\Z_k)^\ell$. 

For simplicity of notation, set 
$N^{(j)} = N(j,(0,\ldots,0)) = N + 2j\lambda_k$, $j \in \Z_k$.
We consider a coset $N(\xi)$ of $N^\ell$ in $(N^\circ)^{\ell}$ defined by 
\begin{equation}\label{N_xi}
  N(\xi) = \{ (x_1,\ldots,x_\ell) \mid x_r \in N^{(\xi_r)}, 1 \le r \le \ell \} 
  \subset (N^\circ)^{\ell}
\end{equation}
for $\xi = (\xi_1, \ldots, \xi_\ell) \in (\Z_k)^\ell$. 
Since $\la \al,\be \ra \in -2ij/k + 2\Z$ for $\al \in N^{(i)}$, 
$\beta \in N^{(j)}$, 
we have
\begin{equation}\label{eq:inner_prod_in_GammaD}
  \la \al,\be \ra \in - \frac{2}{k}(\xi | \eta) + 2\Z 
  \quad \text{for } \al \in N(\xi), \be \in N(\eta).
\end{equation}

Let $D$ be a $\Z_k$-code of length $\ell$.
We consider two cases.

\medskip\noindent
{\bfseries Case A.} \  
$(\xi | \xi) = 0$ for all $\xi \in D$.

\medskip\noindent
{\bfseries Case B.} \ 
$k$ is even, 
$(\xi | \eta) \in \{0, k/2\}$ for all $\xi, \eta \in D$, and 
$(\xi | \xi) = k/2$ for some $\xi \in D$.

\medskip
Let
\begin{equation}\label{eq:GammaD}
\Gamma_D = \bigcup_{\xi \in D} N(\xi) \subset (N^\circ)^{\ell}, 
\end{equation}
which is a sublattice of $(N^\circ)^{\ell}$, 
as $N(\xi) + N(\eta) = N(\xi + \eta)$ and $D$ is an additive subgroup of $(\Z_k)^\ell$. 
The following lemma holds by \eqref{eq:inner_prod_in_GammaD}.

\begin{lemma}\label{lem:GD}
\textup{(1)} $\Gamma_D$ is a positive definite even lattice 
if and only if $D$ is in Case A.

\textup{(2)} $\Gamma_D$ is a positive definite odd lattice if and only if 
If $k$ is even and $D$ is in Case B. 
\end{lemma}

If $D$ is in Case A, then $V_{\Gamma_D}$ is a vertex operator algebra.
If $k$ is even and $D$ is in Case B, 
we set 
\[
D^0 = \{ \xi \in D \mid (\xi | \xi) = 0 \}, \quad
D^1 = \{ \xi \in D \mid (\xi | \xi) = k/2 \}.
\]
We also set $\Gamma_{D^p} =  \bigcup_{\xi \in D^p} N(\xi)$, $p = 0,1$.
Then $D^0$ is a subgroup of the additive group $D$ of index two, and 
$D = D^0 \cup D^1$ is the coset decomposition of $D$ by $D^0$. 
Moreover, 
$\Gamma_{D^p} = \{ \al \in \Gamma_D \mid \la \al, \al \ra \in p + 2\Z \}$, $p = 0,1$, 
and $\Gamma_D = \Gamma_{D^0} \cup \Gamma_{D^1}$   
with $\Gamma_{D^0}$ an even sublattice. 
We have that  
$V_{\Gamma_D} = V_{\Gamma_{D^0}} \oplus V_{\Gamma_{D^1}}$ 
is a vertex operator superalgebra.  

It follows from \eqref{N_xi} that
$V_{N(\xi)} = V_{N^{(\xi_1)}} \otimes \cdots \otimes V_{N^{(\xi_\ell)}} 
\subset (V_{N^\circ})^{\ell}$.
We also have 
$V_{\Gamma_D} = \bigoplus_{\xi \in D} V_{N(\xi)}$ by \eqref{eq:GammaD}.
Let 
\begin{equation*}
  M_\xi = \{ v \in V_{N(\xi)} \mid (\om_{T^{\otimes \ell}})_{(1)} v = 0\},
\end{equation*}
where 
$\om_{T^{\otimes \ell}}$ is the conformal vector of the 
vertex operator subalgebra $T^{\otimes \ell}$ of $(V_N)^{\otimes \ell}$.
Then 
$M_\xi = M^{\xi_1} \otimes \cdots \otimes M^{\xi_\ell}$ 
for $\xi = (\xi_1, \ldots, \xi_\ell) \in (\Z_k)^\ell$ by Proposition \ref{prop:Mj_in_V_Nj0}, 
which is a simple current for
$M_\0 =  (M^{0})^{\otimes \ell}$ 
with $\0 = (0,\ldots,0)$ the zero codeword. 
Since $u_{(n)} v \in V_{N(\xi + \eta)}$ for $u \in V_{N(\xi)}$, $v \in V_{N(\eta)}$, $n \in \Z$, 
we have 
$u_{(n)} v \in M_{\xi + \eta}$ for $u \in M_\xi$, $v \in M_\eta$, $n \in \Z$.
Thus $M_\xi \boxtimes_{M_\0} M_\eta = M_{\xi + \eta}$ for $\xi, \eta \in (\Z_k)^\ell$, and 
$\SC{M_{\0}} = \{ M_\xi \mid \xi \in (\Z_k)^\ell \}$ is $(\Z_k)^\ell$-graded. 
The top level of $M_\xi$ is one dimensional with  
$h(M_\xi) = \big( \sum_{r=1}^\ell \xi_r \big) - (\xi | \xi)/k$, 
as $h(M^j) = j - j^2/k$, 
where $\xi_r$ and $(\xi | \xi)$ are considered to be nonnegative integers.

We have the next proposition by the properties of $M^0 = K(\mfsl_2,k)$ in 
Section \ref{sec:Ksl2k}. 

\begin{proposition}\label{prop:properties_M_D}
$M_\0 =  (M^{0})^{\otimes \ell}$ is a simple, self-dual, rational, and $C_2$-cofinite 
vertex operator algebra of CFT-type with central charge $2\ell (k-1)/(k+2)$. 
Any irreducible $M_\0$-module except for $M_\0$ itself has positive conformal weight.
\end{proposition}

Let $M_D$ be the commutant of $T^{\otimes \ell}$ in $V_{\Gamma_D}$. 
Then
\begin{equation}\label{eq:M_D}
  M_D = \{ v \in V_{\Gamma_D} \mid (\om_{T^{\otimes \ell}})_{(1)} v = 0\} 
  = \bigoplus_{\xi \in D} M_\xi, 
\end{equation}
which is a $D$-graded simple current extension of $M_\0$. 
The following theorem holds.
\begin{theorem}\label{thm:MD}
\textup{(1)} 
If $D$ is in Case A, then $M_D$ is a simple, self-dual, rational, and $C_2$-cofinite vertex 
operator algebra of CFT-type with central charge $2\ell (k-1)/(k+2)$. 

\textup{(2)} 
If $k$ is even and $D$ is in Case B, 
then $M_D = M_{D^0} \oplus M_{D^1}$ is a simple vertex operator superalgebra, 
whose even part $M_{D^0}$ and odd part $M_{D^1}$ are given by 
$M_{D^p} = \bigoplus_{\xi \in D^p} M_\xi$, $p = 0,1$, 
and $h(M_{D^1}) \in \Z + 1/2$.  
\end{theorem}

\section{Irreducible $M_D$-modules: Case A}\label{sec:irred_M_D-modules_A}

Let $k \ge 2$, and let $D$ be a $\Z_k$-code of length $\ell$ 
satisfying the condition of Case A in Section \ref{sec:gamma_D-M_D}, 
that is, $(\xi | \xi) = 0$ for all $\xi \in D$. 
In this section, we classify the irreducible $\chi$-twisted $M_D$-modules for 
$\chi \in D^*$. 
We construct all irreducible untwisted $M_D$-modules 
inside $V_{(\Gamma_D)^\circ}$ as well.

\subsection{Linear characters of $D$}\label{subsec:linear_character_of_D}
 
Let 
\begin{equation*}
P(i,j) = k(i - 2j) - (i - 2j)^2 + 2k(i - j + 1)j. 
\end{equation*}
Then $h(M^{i,j}) = P(i,j)/2k(k+2)$ for $0 \le j \le i \le k$ 
by \eqref{eq:top-weight-Mij}. 
In the case where $0 \le i \le j < k$, 
we have $h(M^{i,j}) = P(k-i,j-i)/2k(k+2)$ by \eqref{eq:equiv-Mij}.
We calculate the values of the map 
$b_{M^0} : \SC{M^0} \times \Irr(M^0) \to \Q/\Z$ defined by  
\begin{equation*}
  b_{M^0}(M^p, M^{i,j}) = h(M^p \boxtimes_{M^0} M^{i,j}) - h(M^p) - h(M^{i,j}) + \Z,
\end{equation*} 
where $M^p \boxtimes_{M^0} M^{i,j} = M^{i,j+p}$ by \eqref{eq:fusion_rule_for_M0}.
If $0 \le j < i \le k$, then $0 \le j < j+1 \le i \le k$, and
\begin{equation*}
  P(i,j+1) - P(i,j) = 2(k+2)(i-2j-1), 
\end{equation*}
whereas if $0 \le i \le j < k$, then $0 \le j-i < j+1-i \le k-i \le k$, 
and
\begin{equation*}
  P(k-i,j+1-i) -  P(k-i,j-i) = 2(k+2)(i-2j+k-1).
\end{equation*}

In both cases, we have $b_{M^0}(M^1, M^{i,j}) = (i-2j)/k + \Z$. 
Thus
\begin{equation}\label{eq:b_paraf}
b_{M^0}(M^p, M^{i,j}) = \frac{p(i-2j)}{k} + \Z
\end{equation}
for $0 \le i \le k$, $0 \le j < k$, and  $0 \le p < k$ 
by Lemma \ref{lem:bV-bilinear-rev}.  

For $\mu = (\mu_1,\ldots,\mu_\ell)$ with $0 \le \mu_r \le k$, $1 \le r \le \ell$, 
and $\nu = (\nu_1,\ldots, \nu_\ell) \in (\Z_k)^\ell$,  
let
\begin{equation*}
  M_{\mu,\nu} = M^{\mu_1,\nu_1} \otimes \cdots \otimes M^{\mu_\ell,\nu_\ell}.
\end{equation*}
Then 
\begin{equation}\label{eq:Irr_M_0}
  \Irr(M_{\0}) 
  = \{ M_{\mu,\nu} \mid \mu = (\mu_1,\ldots,\mu_\ell), 
  \nu = (\nu_1,\ldots, \nu_\ell), 
  0 \le \nu_r < \mu_r \le k, 1 \le r \le \ell \}.
\end{equation}
It follows from \eqref{eq:fusion_rule_for_M0} that 
\begin{equation}\label{eq:M_xi_times}
  M_\xi \boxtimes_{M_\0} M_{\mu,\nu} = M_{\mu,\nu + \xi}.
\end{equation}

Let $b_{M_{\0}} : \SC{M_{\0}} \times \Irr(M_{\0}) \to \Q/\Z$ be a map defined by  
\begin{equation*}
  b_{M_{\0}}(M_\xi, M_{\mu,\nu}) = 
  h(M_\xi \boxtimes_{M_\0} M_{\mu,\nu}) - h(M_\xi) - h(M_{\mu,\nu}) + \Z
\end{equation*}
for $\mu = (\mu_1,\ldots,\mu_\ell)$ with $0 \le \mu_r \le k$, 
$1 \le r \le \ell$,   
$\nu = (\nu_1,\ldots, \nu_\ell) \in (\Z_k)^\ell$, 
and $\xi = (\xi_1,\ldots, \xi_\ell) \in (\Z_k)^\ell$.
Then 
\eqref{eq:b_paraf} implies that
\begin{equation}\label{eq:b_MD-1}
  b_{M_\0}(M_\xi,M_{\mu,\nu}) = \sum_{r=1}^\ell  \frac{\xi_r(\mu_r - 2\nu_r)}{k} + \Z.
\end{equation}

Although $\mu_r$ is an integer between $0$ and $k$, 
we can treat $\mu_r$ modulo $k$ on the right hand side of \eqref{eq:b_MD-1}. 
Then \eqref{eq:b_MD-1} is written as
\begin{equation}\label{eq:b_MD-2}
  b_{M_\0}(M_\xi, M_{\mu,\nu}) = \frac{1}{k}(\xi | \mu - 2\nu) + \Z,
\end{equation}
where $(\, \cdot \,|\, \cdot \,)$ is the standard inner product on $(\Z_k)^\ell$. 
In particular,
\begin{equation}\label{eq:b_MD-3}
  b_{M_\0}(M_\xi, M_\eta) = - \frac{2}{k}(\xi | \eta) + \Z.
\end{equation}

\begin{lemma}\label{lem:bM0}
Let $\xi, \eta, \nu \in (\Z_k)^\ell$, 
and let $\mu = (\mu_1,\ldots,\mu_\ell)$ with $0 \le \mu_r \le k$, $1 \le r \le \ell$.

\textup{(1)} $b_{M_\0}(M_\xi, M_\eta) = 0$ if $\xi, \eta \in D$.

\textup{(2)} $b_{M_\0}(M_{\xi+\eta}, M_{\mu, \nu}) 
= b_{M_\0}(M_\xi, M_{\mu, \nu}) + b_{M_\0}(M_\eta, M_{\mu, \nu})$.

\textup{(3)} $b_{M_\0}(M_\xi, M_{\mu, \nu+\eta}) 
= b_{M_\0}(M_\xi, M_\eta) + b_{M_\0}(M_\xi, M_{\mu, \nu})$. 
\end{lemma}

\begin{proof}
Suppose $\xi, \eta \in D$. 
Then $\xi + \eta \in D$, so $(\xi+\eta | \xi+\eta) = 0$ by our assumption on $D$. 
Since $(\xi | \xi) = (\eta | \eta) = 0$, we have $2(\xi | \eta) = 0$. 
Thus the assertion (1) holds by \eqref{eq:b_MD-3}. 
The assertions (2) and (3) are clear from \eqref{eq:b_MD-2}, 
see also Lemma \ref{lem:bV-bilinear-rev}. 
\end{proof}

For $\eta \in (\Z_k)^\ell$,  
let $\chi(\eta)$ be a linear character of 
the abelian group $(\Z_k)^\ell$ given by 
\begin{equation*}
  \chi(\eta) : (\Z_k)^\ell \to \C^\times; \quad 
  \xi \mapsto \exp(2\pi\sqrt{-1} (\xi | \eta)/k).
\end{equation*}
Then $(\Z_k)^\ell \to \Hom((\Z_k)^\ell, \C^\times)$;  
$\eta \mapsto \chi(\eta)$ is a group isomorphism.
The linear character $\chi_{M_{\mu, \nu}} \in D^\ast$ is 
the restriction $\chi(\mu-2\nu)|_D$ of $\chi(\mu-2\nu)$ to $D$ 
by \eqref{eq:b_MD-2}.  That is,
\begin{equation}\label{eq:char_of_D}
  \chi_{M_{\mu, \nu}} (\xi) 
  = \exp(2\pi\sqrt{-1} b_{M_\0}(M_\xi, M_{\mu, \nu}))
  = \exp(2\pi\sqrt{-1} (\xi | \mu - 2\nu)/k).
\end{equation}

Let $D^\perp = \{ \eta \in (\Z_k)^\ell \mid (D | \eta) = 0\}$.
Then $\abs{D} \abs{D^\perp} = \abs{(\Z_k)^\ell}$, as $(\, \cdot \,|\, \cdot \,)$ 
is a non-degenerate bilinear form. 

\begin{lemma}\label{lem:linear_char_of_D}
\textup{(1)} The map $(\Z_k)^\ell \to D^\ast$; $\eta \mapsto \chi(\eta)|_D$ 
is a surjective group homomorphism with kernel $D^\perp$. 

\textup{(2)} For any $\chi \in D^\ast$, 
there exists $M_{\mu, \0} \in \Irr(M_\0)$ such that $\chi = \chi_{M_{\mu, \0}}$. 

\textup{(3)} $\chi_{M_{\mu, \nu}} = 1$; the principal character of $D$ if and only if 
$\mu - 2\nu \in D^\perp$.

\textup{(4)} $\chi_{M_{\mu, \nu}} = \chi_{M_{\mu', \nu'}}$ if and only if 
$\mu - 2\nu \equiv \mu' - 2\nu' \pmod{D^\perp}$.
\end{lemma}

\begin{proof}
Non-degeneracy of the bilinear form $(\, \cdot \,|\, \cdot \,)$ implies 
the assertions (1) and (2). 
The assertions (3) and (4) are consequences of 
\eqref{eq:char_of_D} and the definition of $D^\perp$.
\end{proof}

\subsection{Irreducible $M_\0$-modules in $V_{(N^\circ)^\ell}$}
\label{subsec:irred_M_0-modules}

Let 
\begin{equation*}
  N(\eta, \delta) = \{ (x_1, \ldots, x_\ell) \mid x_r \in N(\eta_r, (0, \ldots, 0, d_r)), 
  1 \le r \le \ell \} \subset (N^\circ)^\ell
\end{equation*}
for $\eta = (\eta_1, \ldots, \eta_\ell) \in (\Z_k)^\ell$ and 
$\delta = (d_1,\ldots,d_\ell) \in \{0,1\}^\ell$. 

\begin{proposition}\label{prop:irred_M0-mod_in_coset}
\textup{(1)} 
Let $\eta = (\eta_1, \ldots, \eta_\ell) \in (\Z_k)^\ell$ and 
$\delta = (d_1,\ldots,d_\ell) \in \{0,1\}^\ell$. 
Assume that $\mu = (\mu_1,\ldots,\mu_\ell)$ with $0 \le \mu_r \le k$, $1 \le r \le \ell$, and $\nu = (\nu_1,\ldots,\nu_\ell) \in (\Z_k)^\ell$ satisfy the conditions 
\begin{equation}\label{eq:condition_on_eta_delta}
\mu_r \equiv d_r \pmod{2}, \quad \nu_r = \eta_r + \frac{\mu_r - d_r}{2}, \quad 
1 \le r \le \ell.
\end{equation}
Then 
$V_{N(\eta, \delta)}$ contains the irreducible $M_\0$-module $M_{\mu, \nu}$.

\textup{(2)}
Any irreducible $M_\0$-module is contained in $V_{N(\eta, \delta)}$ for some 
$\eta$ and $\delta$. 
If $k$ is odd, then we can choose $\delta$ to be $\delta = (0,\ldots,0)$.
\end{proposition}

\begin{proof}
The assertions (1) and (2) hold by Propositions \ref{prop:Mij_in_V_Nj0d} and  
\ref{prop:appearance_of_Mij}. 
\end{proof}

\begin{lemma}\label{inner_product_cosets}
Let $\xi, \eta \in (\Z_k)^\ell$ and $\delta \in \{0,1\}^\ell$. 
Then $\la x, y \ra \in (\xi | \delta - 2\eta)/k + \Z$ for 
$x \in N(\xi)$ and $y \in N(\eta,\delta)$.
\end{lemma}

\begin{proof}
Since $\la x, y \ra \in p(d-2j)/k + \Z$ for $x \in N^{(p)}$ 
and $y \in N(j,(0,\ldots,0,d))$, the assertion holds.
\end{proof}

\begin{proposition}\label{prop:untwisted_coset}
Let $\mu = (\mu_1,\ldots,\mu_\ell)$ with $0 \le \mu_r \le k$, $1 \le r \le \ell$, 
and let $\nu = (\nu_1,\ldots,\nu_\ell) \in (\Z_k)^\ell$. 
Take $\eta \in (\Z_k)^\ell$ and 
$\delta \in \{0,1\}^\ell$ such that 
the conditions \eqref{eq:condition_on_eta_delta} hold. 
Then $b_{M_\0}(M_\xi, M_{\mu, \nu}) = 0$ for all $\xi \in D$ if and only if 
$N(\eta, \delta) \subset (\Gamma_D)^\circ$. 
\end{proposition}

\begin{proof}
Since $\mu_r - 2\nu_r = d_r - 2\eta_r$ by \eqref{eq:condition_on_eta_delta}, 
the assertion holds by \eqref{eq:b_MD-2} and 
Lemma \ref{inner_product_cosets}.
\end{proof}

\subsection{Irreducible twisted $M_D$-modules in $V_{(N^\circ)^\ell}$}
\label{subsec:irred_M_D-modules}

Let $X \in \Irr(M_{\0})$. 
Then $X = M_{\mu,\nu}$ for some $\mu$ and $\nu$ by \eqref{eq:Irr_M_0}. 
Take $\eta$ and $\delta$ such that the conditions 
\eqref{eq:condition_on_eta_delta} hold. 
Then $V_{N(\eta, \delta)}$ contains $M_{\mu,\nu}$ as an $M_{\0}$-submodule 
by Proposition \ref{prop:irred_M0-mod_in_coset}. 
Since $M_\xi \subset V_{N(\xi)}$,  
and since $N(\xi) + N(\eta,\delta) = N(\xi+\eta,\delta)$, 
it follows that $V_{N(\xi+\eta,\delta)}$ 
contains $M_\xi \boxtimes_{M_{\0}} M_{\mu,\nu}$.  
For fixed $\eta$ and $\delta$, the cosets $N(\xi+\eta,\delta)$, $\xi \in D$,  
of $N^\ell$ in $(N^\circ)^\ell$ are all distinct. 
Hence the $\chi_{M_{\mu,\nu}}$-twisted $M_D$-module $M_D \cdot M_{\mu,\nu}$ 
generated by $M_{\mu,\nu}$ in $V_{(N^\circ)^\ell}$ is isomorphic to 
$M_D\boxtimes_{M_{\0}} {M_{\mu,\nu}}$ 
by (2) of Theorem \ref{thm:induced-module}.
Furthermore, if $\chi_{M_{\mu,\nu}}(\xi) = 1$ for all $\xi \in D$, then 
$N(\eta, \delta) \subset (\Gamma_D)^\circ$ 
by Proposition \ref{prop:untwisted_coset}, 
so $M_D \cdot M_{\mu,\nu} \subset V_{(\Gamma_D)^\circ}$. 
Therefore, the following theorem holds.

\begin{theorem}\label{thm:contain_univ}
Let $X \in \Irr(M_{\0})$. 

\textup{(1)} 
$V_{(N^\circ)^\ell}$ contains a $\chi_X$-twisted $M_D$-module isomorphic to 
$M_D\boxtimes_{M_{\0}} X$. 

\textup{(2)} 
If $\chi_X = 1$, then $V_{(\Gamma_D)^\circ}$ contains 
an untwisted $M_D$-module isomorphic to 
$M_D\boxtimes_{M_{\0}} X$.
\end{theorem}

Let $W$ be an irreducible $\chi$-twisted $M_D$-module for  $\chi \in D^*$,  
and let $X$ be an irreducible $M_\0$-submodule of $W$. 
Then $W$ is isomorphic to a direct summand of
$M_D\boxtimes_{M_{\0}} X$ 
with $\chi = \chi_X$ by (3) of Theorem \ref{thm:induced-module}. 
Thus Theorem \ref{thm:contain_univ} implies the following theorem.

\begin{theorem}\label{thm:contain_irred}
\textup{(1)} 
$V_{(N^\circ)^\ell}$ contains any irreducible $\chi$-twisted $M_D$-module 
for $\chi \in D^*$.

\textup{(2)} 
$V_{(\Gamma_D)^\circ}$ contains any irreducible untwisted $M_D$-module.
\end{theorem}

Let $\Irr(M_\0) = \bigcup_{i\in I} \orbit_i$ 
be the $D$-orbit decomposition of $\Irr(M_\0)$ for the action of $D$ on 
$\Irr(M_\0)$ in \eqref{eq:M_xi_times},  
and let 
$D_{M_{\mu,\nu}} 
= \{ \xi \in D \mid M_\xi \boxtimes_{M_{\0}} M_{\mu,\nu} \cong M_{\mu,\nu} \}$
be the stabilizer of $M_{\mu,\nu}$.
Lemma \ref{eq:exceptional-id} implies the following lemma.

\begin{lemma}\label{lem:exceptional_id_M0}
$M_\xi \boxtimes_{M_{\0}} M_{\mu,\nu} \cong M_{\mu,\nu}$  
as $M_{\0}$-modules for some $\xi \ne \0$ if and only if 
$k$ is even, $\xi = (\xi_1,\ldots,\xi_\ell) \in \{0, k/2\}^\ell$, and 
$\mu_r = k/2$ for $1 \le r \le \ell$ such that $\xi_r = k/2$. 
\end{lemma}

The next theorem is a restatement of Proposition \ref{prop:stable-module}.

\begin{theorem}\label{thm:stable-module_M0}
  Let $X \in \Irr(M_\0)$. 
  If $D_X = 0$, then
  $M_D \boxtimes_{M_{\0}} X$ is an irreducible 
  $\chi_X$-twisted $M_D$-module.
\end{theorem}

Now, suppose $D_X \ne 0$. 
Then $k$ is even and $D_X \subset \{0, k/2\}^\ell$ 
by Lemma \ref{lem:exceptional_id_M0}. 
In order to apply the results in Section \ref{subsec:irred_VD-modules}, 
we recall the previous arguments for a special case where 
the $\Z_k$-code is of length one consisting of two codewords $(0)$ and $(k/2)$. 
Let $C = \{ (0), (k/2) \}$ be such a $\Z_k$-code. 
Then $\Gamma_C = N \cup N^{(k/2)}$ with $N^{(k/2)} = N + k \lambda_k$, 
and $M_C = M^0 \oplus M^{k/2}$ is a $\Z_2$-graded simple current extension 
of $M^0$ by the self-dual simple current $M^0$-module $M^{k/2}$ 
with $h(M^{k/2}) = k/4$. 

If $k \equiv 0 \pmod{4}$, 
then $(k/2)^2 \equiv 0 \pmod{k}$. 
Hence the $\Z_k$-code $C$ is in Case A in Section \ref{sec:gamma_D-M_D},  
and $M_C$ is a simple vertex operator algebra with 
$h(M^{k/2}) \in \Z$. 
If $k \equiv 2 \pmod{4}$, 
then $(k/2)^2 \equiv k/2 \pmod{k}$. 
Hence $C$ is in Case B in Section \ref{sec:gamma_D-M_D}, 
and $M_C$ is a simple vertex operator superalgebra 
with $h(M^{k/2}) \in \Z + 1/2$. 
In both cases, there exists a unique structure of a $\Z_2$-graded 
either untwisted or $\Z_2$-twisted $M_C$-module on the space 
$P^0 \oplus P^1$ with $P^0 = P$ and $P^1 = M^{k/2} \boxtimes_{M^0} P$ 
for any irreducible $M^0$-module $P$. 

Under the correspondence $0 \mapsto 0$ and $k/2 \mapsto 1$, 
we can regard any additive subgroup of $\{ 0, k/2\}^\ell \subset (\Z_k)^\ell$ 
as an additive subgroup of $(\Z_2)^\ell$. 
In the case where $k \equiv 2 \pmod{4}$, this correspondence is 
the reduction modulo $2$, and it in fact gives an isometry from 
$(\{ 0, k/2\}^\ell, (\, \cdot\, | \, \cdot \,))$ to 
$((\Z_2)^\ell, (\, \cdot\, | \, \cdot \,))$, where $(\, \cdot\, | \, \cdot \,)$ is 
the standard inner product on either $(\Z_k)^\ell$ or $(\Z_2)^\ell$. 
Hence $D_X \cap D_X^\perp$ in $(\Z_k)^\ell$ corresponds to 
$D_X \cap D_X^\perp$ in $(\Z_2)^\ell$. 
Therefore, we obtain the following theorem by  
Propositions \ref{prop:decomp1} and \ref{prop:decomp2}.

\begin{theorem}\label{thm:decomp_M_D_X}
Let $X \in \Irr(M_\0)$. 
Suppose $k$ is even and $D_X \ne 0$.

\textup{(1)} If $k \equiv 0 \pmod{4}$, then
the irreducible decomposition of the $\chi_X$-twisted $M_D$-module 
$M_D \boxtimes_{M_\0} X$ is given as 
\[
  M_D \boxtimes_{M_\0} X = \bigoplus_{j = 1}^{\abs{D_X}} U^j,
\]
where $U^j$, $1 \leq j \leq \abs{D_X}$, are inequivalent irreducible 
$\chi_X$-twisted $M_D$-modules. 
Furthermore, 
$U^j \cong \bigoplus_{W \in \orbit_i} W$ 
as $M_\0$-modules, where $\orbit_i$ is the $D$-orbit in $\Irr(M_\0)$ 
containing $X$.

\textup{(2)} If $k \equiv 2 \pmod{4}$, then  
the irreducible decomposition of the $\chi_X$-twisted $M_D$-module 
$M_D \boxtimes_{M_\0} X$ is given as 
\[
  M_D \boxtimes_{M_\0} X 
  = \bigoplus_{j=1}^{\abs{D_X \cap D_X^\perp}} (U^j)^{\oplus m}, 
\]
where $m = [D_X : D_X \cap D_X^\perp]^{1/2}$, and 
$U^j$, $1 \leq j \leq \abs{D_X \cap D_X^\perp}$, are inequivalent irreducible 
$\chi_X$-twisted $M_D$-modules. 
Furthermore, 
$U^j \cong \bigoplus_{W \in \orbit_i} W^{\oplus m}$ 
as $M_\0$-modules, where $\orbit_i$ is the $D$-orbit in $\Irr(M_\0)$ 
containing $X$.
\end{theorem}

Since any irreducible $\chi$-twisted $M_D$-module for $\chi \in D^*$  
is isomorphic to a direct summand 
of the $\chi_X$-twisted $M_D$-module 
$M_D\boxtimes_{M_{\0}} X$ with $\chi = \chi_X$ for some $X \in \Irr(M_\0)$, 
we obtain a classification of all the irreducible $\chi$-twisted $M_D$-modules 
for any $\chi \in D^*$ by 
Theorems \ref{thm:stable-module_M0} and \ref{thm:decomp_M_D_X}. 

As mentioned in Section \ref{subsec:irred_VD-modules}, 
we can write $\chi_i$ for $\chi_X$, and $D_i$ for $D_X$ 
if $X$ belongs to a $D$-orbit $\orbit_i$ in $\Irr(M_\0)$.
Let $I(\chi) = \{ i \in I \mid \chi_i = \chi \}$. 
Then $I = \bigcup_{\chi \in D^*} I(\chi)$. 
The next lemma follows from (2) of Lemma \ref{lem:linear_char_of_D}. 

\begin{lemma}
  $I(\chi) \ne \varnothing$ for any $\chi \in D^*$.
\end{lemma}

Theorems \ref{thm:stable-module_M0} and \ref{thm:decomp_M_D_X} imply 
the next theorem.

\begin{theorem}\label{thm:count_irred_twisted_mod}
The number of inequivalent irreducible 
$\chi$-twisted $M_D$-modules for $\chi \in D^*$ is given as follows.
\begin{alignat*}{2}
& \abs{I(\chi)} & \quad 
& \text{if $k$ is odd},\\
& \abs{I(\chi)_0} + \sum_{i \in I(\chi)_1} \abs{D_i} & \quad 
& \text{if } k \equiv 0 \pmod{4},\\
& \abs{I(\chi)_0} + \sum_{i \in I(\chi)_1} \abs{D_i \cap D_i^\perp} & \quad 
& \text{if } k \equiv 2 \pmod{4}, 
\end{alignat*}
where $I(\chi)_0 = \{ i \in I(\chi) \mid D_i = 0 \}$ and 
$I(\chi)_1 = I(\chi) \setminus I(\chi)_0$. 
\end{theorem}

\section{Irreducible $M_D$-modules: Case B}\label{sec:irred_M_D-modules_B}

Let $k \ge 2$, and let $D$ be a $\Z_k$-code of length $\ell$ 
satisfying the conditions of Case B in Section \ref{sec:gamma_D-M_D}, 
that is, 
$k$ is even, 
$(\xi | \eta) \in \{0, k/2\}$ for all $\xi, \eta \in D$, and 
$(\xi | \xi) = k/2$ for some $\xi \in D$. 
Let $D^0$ and $D^1$ be as in Section \ref{sec:gamma_D-M_D}. 
In this section, 
we construct all irreducible $M_D$-modules 
inside $V_{(\Gamma_{D^0})^\circ}$. 

Since $D^0$ is a $\Z_k$-code of length $\ell$ in Case A, we can apply the results 
in Section \ref{sec:irred_M_D-modules_A} to the vertex operator algebra $M_{D^0}$. 
Let $P \in \Irr(M_{D^0})$. 
Then $P$ is isomorphic to a direct summand of $M_{D^0} \boxtimes_{M_\0} M_{\mu,\nu}$ 
for some $M_{\mu,\nu} \in \Irr(M_\0)$. 
Moreover, there are $\eta \in (\Z_k)^\ell$ and $\delta \in \{0,1\}^\ell$ such that 
$N(\eta, \delta) \subset (\Gamma_{D^0})^\circ$ 
and $V_{N(\eta, \delta)}$ contains $M_{\mu,\nu}$ as an $M_{\0}$-submodule. 

For simplicity of notation, we identify $P$ with an irreducible direct summand of 
$M_{D^0} \boxtimes_{M_\0} M_{\mu,\nu}$ isomorphic to $P$. 
Then $P$ is a submodule of the $M_{D^0}$-module $M_{D^0} \boxtimes_{M_\0} M_{\mu,\nu}$, 
and the $M_D$-module $M_D \cdot P$ generated by $P$ is isomorphic to 
$M_D \boxtimes_{M_{D^0}} P$. 
Thus $M_D \cdot P = P \oplus Q$ as $M_{D^0}$-modules, 
where $Q$ is an irreducible $M_{D^0}$-module isomorphic to $M_{D^1} \boxtimes_{M_{D^0}} P$. 
Since $\Gamma_D \subset (\Gamma_D)^\circ \subset (\Gamma_{D^0})^\circ$, 
and since $M_{\mu,\nu} \subset V_{(\Gamma_{D^0})^\circ}$, 
we have $M_D \cdot P \subset V_{(\Gamma_{D^0})^\circ}$. 

If $P$ and $Q$ are inequivalent as $M_{D^0}$-modules, 
then there is a unique $M_D$-module structure on $P \oplus Q$ which extends the 
$M_{D^0}$-module structure. 
If $P$ and $Q$ are equivalent as $M_{D^0}$-modules, then $P \oplus Q$ is the direct sum of 
two inequivalent irreducible $M_D$-modules, both of which are isomorphic to $P$ as 
$M_{D^0}$-modules, see \cite[Proposition 5.2]{Li1997}. 
Any irreducible $M_D$-module is obtained in this way. 
Therefore, the following theorem holds.

\begin{theorem}\label{thm:contain_irred_B}
$V_{(\Gamma_{D^0})^\circ}$ contains any irreducible $M_D$-module.
\end{theorem}

\section{Examples}\label{sec:examples}

The vertex operator algebra $M_D$ was previously studied for some small $k$. 
The first one is the case $k = 2$, where $M^0$ is the Virasoro vertex operator 
algebra $L(1/2,0)$ of central charge $1/2$, and its simple currents are 
$M^0$ and $M^1 = L(1/2,1/2)$.  
The next one is the case $k = 3$, where $M^0$ is $L(4/5,0) \oplus L(4/5,3)$,  
and there are three simple currents. 
These cases were discussed in \cite{Miyamoto1996} and \cite{KMY2000}, respectively.  
 
In the case $k = 4$, we have $M^0 = V_{\sqrt{6}\Z}^+$ and $M^2 = V_{\sqrt{6}\Z}^-$. 
So $M_D = V_{\sqrt{6}\Z}$ for $\ell = 1$ and $D = \{ (0), (2) \}$. 
The case $k = 5$ with $\ell = 2$ and $D = \{ (00), (12), (24), (31), (43) \}$, and 
the case $k = 9$ with $\ell = 1$ and $D = \{(0), (3), (6) \}$ were considered  
in Sections 3.5 and 3.9 of \cite{LYY2005}, respectively.

Let $k=6$ with $\ell = 1$ and $D = \{ (0), (3) \}$. 
Then 
\begin{equation*}
M_D = M^0 \oplus M^3 \cong L_{\mathrm{NS}}(5/4,0) \oplus L_{\mathrm{NS}}(5/4,3),
\end{equation*}
where $L_{\mathrm{NS}}(5/4,0)$ is the simple Neveu-Schwarz algebra of central charge $5/4$,  
and $L_{\mathrm{NS}}(5/4,3)$ is its irreducible highest weight module with highest weight $3$, see \cite[Section 4]{AYY2016}, \cite{ZF1985}. 
In fact, let $v$ be an weight $3/2$ element of $M^3$ such that $v_{(2)} v = (5/6)\1$. 
Then $L_n = \omega_{(n+1)}$ and $G_{n - 1/2} = v_{(n)}$,  $n \in \Z$, 
satisfy the relations for the Neveu-Schwarz algebra of central charge $5/4$. 
Thus the subalgebra generated by $\omega$ and $v$ in $V_{\Gamma_D}$ 
is isomorphic to $L_{\mathrm{NS}}(5/4,0)$. 
Moreover, the weight $3$ primary vector $W^3$ of $M^0$ 
is a highest weight vector for $L_{\mathrm{NS}}(5/4,0)$.

\appendix

\section{minimal norm of elements in $N(j,\bsa)$}

In this appendix, we calculate the minimal norm of elements in  
the coset $N(j,\bsa)$ of $N$ in $N^\circ$ defined in \eqref{def:N_j_a}. 
Let $\Omega = \{1,2, \ldots, k\}$, 
and let $\al_S = \sum_{p \in S} \al_p$ for a subset $S$ of $\Omega$.

\begin{theorem}\label{thm:minimal_norm}
Let $\bsa \in \{ 0,1 \}^{k}$ and $0 \le j \le k-1$. 
Set $I = \supp(\bsa)$ and $i = \wt(\bsa)$.

\textup{(1)} If $j < i$, then

\quad \textup{(i)} 
$\min \{ \la \mu,\mu \ra \mid \mu \in N(j,\bsa)\} = ( k i - (i-2j)^2)/2k$,

\quad \textup{(ii)} For $\mu \in N(j,\bsa)$, 
the norm $\la \mu,\mu \ra$ is minimal if and only if 
\begin{equation*}
  \mu = \frac{1}{2} \al_I - \al_J + \frac{2j-i}{2k} \gamma
\end{equation*}

\qquad for some $J \subset I$ with $\abs{J} = j$. 
There are $\binom{i}{j}$ such $\mu$'s. 

\textup{(2)} If $j \ge i$, then

\quad \textup{(i)} 
$\min \{ \la \mu,\mu \ra \mid \mu \in N(j,\bsa)\} = ( k(k-i) - (k+i-2j)^2)/2k$,

\quad \textup{(ii)} For $\mu \in N(j,\bsa)$, 
the norm $\la \mu,\mu \ra$ is minimal if and only if 
\begin{equation*}
  \mu = \frac{1}{2} \al_I - \al_J + \frac{2j-i}{2k} \gamma
\end{equation*}

\qquad for some $I \subset J \subset \Omega$ with $\abs{J} = j$.  
There are $\binom{k-i}{j-i}$ such $\mu$'s. 
\end{theorem}

\begin{proof}
Any permutation on $\{\al_1, \ldots, \al_{k}\}$ 
induces an isometry on $\Q \otimes_{\Z} L$. 
The isometry fixes $\gamma$ and leaves $L$ invariant. 
Since $\lambda_p = \gamma/2k - \alpha_p/2$ 
and $2\lambda_p \equiv 2\lambda_k \pmod{N}$, 
$1 \le p \le k$,  we may assume that $I = \{1, \ldots, i \}$, that is, 
$a_p = 1$ for $p \le i$, and $a_p = 0$ for $p \ge i+1$ in \eqref{def:N_j_a}.

Let $d = (2j-i)/2k$. Since $\al_p \equiv \al_q \pmod{N}$, $1\le p,q \le k$,  
and since any element of $N$ is of the form $c_1\al_1 + \cdots + c_{k}\al_{k}$ 
for some $c_1, \ldots, c_{k} \in \Z$ with $c_1 + \cdots + c_{k} = 0$, 
we see from \eqref{eq:N_j_a-bis} that any element $\mu \in N(j, \bsa)$ is of the form 
\begin{equation*}
\begin{split}
\mu 
&= \frac{1}{2}(\al_1 + \cdots + \al_i) - c_1\al_1 - \cdots - c_{k}\al_{k} + d\gamma\\
&= \sum_{p=1}^i  (d + 1/2 - c_p) \al_p + \sum_{q = i+1}^{k} (d - c_q) \al_q
\end{split}
\end{equation*}
for some $c_1, \ldots, c_{k} \in \Z$ with $c_1 + \cdots + c_{k} = j$.
Our concern is the minimum of 
\begin{equation}\label{eq:mu_square_half}
\la \mu, \mu \ra /2
= \sum_{p=1}^i  (d + 1/2 - c_p)^2 + \sum_{q = i+1}^{k} (d - c_q)^2
\end{equation}
for $c_1, \ldots, c_{k} \in \Z$ with $c_1 + \cdots + c_{k} = j$.

We first show the assertion (1). 
Assume that $0 \le j < i \le k$. Then $-1/2 \le d < 1/2$. 
If $d = -1/2$, then $i = k$ and $j = 0$. 
In this case, we have $N(j, \bsa) = N$. 
Clearly, $\min \{ \la \mu,\mu \ra \mid \mu \in N\} = 0$, and 
$\la \mu,\mu \ra = 0$ only if $\mu = 0$. 
Hence the assertion (1) holds in the case $d = -1/2$. 

If $d = 0$, then $i = 2j$, and \eqref{eq:mu_square_half} reduces to 
$\la \mu, \mu \ra /2
= \sum_{p=1}^i  (1/2 - c_p)^2 + \sum_{q = i+1}^{k} {c_q}^2$.
We see that $(1/2 - c_p)^2$ is $1/4$ if $c_p = 0,1$, and $9/4$ if $c_p = -1, 2$. 
Moreover, ${c_q}^2$ is $0$ if $c_q = 0$, and $1$ if $c_q = \pm 1$. 
Hence the minimum of $\la \mu, \mu \ra /2$ 
for $c_1, \ldots, c_{k} \in \Z$ with $c_1 + \cdots + c_{k} = j$ 
is attained only when $j$ of $c_1, \ldots, c_i$ are $1$, the remaining $i-j$ of 
$c_1, \ldots, c_i$ are $0$, and $c_q = 0$ for $i+1 \le q \le k$.
The minimum of $\la \mu, \mu \ra /2$ is $i/4$. 
Thus the assertion (1) holds in the case $d = 0$.

If $-1/2 < d < 0$, then $0 < d + 1/2 < 1/2$. 
In this case, $(d + 1/2 - c_p)^2$ 
belongs to one of the four open intervals $(0, 1/4)$, $(1/4, 1)$, $(1, 9/4)$, or $(9/4, 4)$ 
according as $c_p = 0$, $1$, $-1$, or $2$, respectively. 
Moreover, $(d - c_q)^2$ 
belongs to one of the four open intervals  $(0, 1/4)$, $(1/4, 1)$, $(1, 9/4)$, or $(9/4, 4)$ 
according as $c_q = 0$, $-1$, $1$, or $-2$, respectively. 
Hence the minimum of \eqref{eq:mu_square_half} for 
$c_1, \ldots, c_{k} \in \Z$ with $c_1 + \cdots + c_{j} = j$ 
is attained only when $j$ of $c_1, \ldots, c_i$ are $1$, the remaining $i-j$ of 
$c_1, \ldots, c_i$ are $0$, and $c_q = 0$ for $i+1 \le q \le k$.
The minimum of \eqref{eq:mu_square_half} is 
\[
  (d - 1/2)^2 j + (d+1/2)^2 (i-j) + d^2 (k - i) = i/4 - (i-2j)^2/4k.
\]
Thus the assertion (1) holds in the case $-1/2 < d < 0$. 

If $0 < d < 1/2$, then $1/2 < d + 1/2 < 1$. 
In this case, $(d + 1/2 - c_p)^2$  
belongs to one of the four open intervals $(0, 1/4)$, $(1/4, 1)$, $(1, 9/4)$, or $(9/4, 4)$ 
according as $c_p = 1$, $0$, $2$, or $-1$, respectively. 
Moreover, $(d - c_q)^2$  
belongs to one of the four open intervals $(0, 1/4)$, $(1/4, 1)$, $(1, 9/4)$, or $(9/4, 4)$ 
according as $c_q = 0$, $1$, $-1$, or $2$, respectively. 
Hence the minimum of \eqref{eq:mu_square_half} for 
$c_1, \ldots, c_{k} \in \Z$ with $c_1 + \cdots + c_{k} = j$ 
is attained only when $j$ of $c_1, \ldots, c_i$ are $1$, the remaining $i-j$ of 
$c_1, \ldots, c_i$ are $0$, and $c_q = 0$ for $i+1 \le q \le k$.
Thus the assertion (1) holds in the case $0 < d < 1/2$. 
We have shown that (1) holds for all $0 \le j < i \le k$.

Next, we show the assertion (2). 
Assume that $j \ge i$. We use Lemma \ref{lem:dual_N_3}. 
Let $a_p' = 1 - a_p$, $1 \le p \le k$, $\bsa' = (a'_1, \ldots, a'_{k})$, and $I' = \supp(\bsa')$. 
Then $I \cup I' = \Omega$ and $I \cap I' = \varnothing$. 
Let $i' = \wt(\bsa')$ and $j' = j - i$. 
Then $i' = k - i$ and $0 \le j' < i' \le k$. 
The assertion (1) for $N(j', \bsa')$ implies that

(i)$'$ $\min \{ \la \mu,\mu \ra \mid \mu \in N(j',\bsa')\} 
= (k i' - (i'-2j')^2)/2k$,

(ii)$'$ For $\mu \in N(j',\bsa')$, the norm $\la \mu,\mu \ra$ is minimal if and only if 
\begin{equation}\label{eq:mu_minimal_1}
\mu = \frac{1}{2} \al_{I'} - \al_{J'} + \frac{2j'-i'}{2k} \gamma
\end{equation}

\quad for some $J' \subset I'$ with $|J'| = j'$. 
There are $\binom{i'}{j'}$ such $\mu$'s. 

Since $\al_{I'} = \gamma - \al_I$, and since 
$2j' - i' = 2j - i - k$, the element $\mu$ of \eqref{eq:mu_minimal_1} 
is equal to
\begin{equation*}
\mu = - \frac{1}{2} \al_I - \al_{J'} + \frac{2j-i}{2k} \gamma.
\end{equation*}

The set $\{ J \subset \Omega \mid I \subset J, |J| = j \}$ 
is in one-to-one correspondence with the set 
$\{ J' \subset \Omega - I \mid |J'| = j - i \}$ 
by $J \mapsto J - I$ and $J' \mapsto J' \cup I$. 
Let $J = J' \cup I$. 
Then $\al_J = \al_{J'} + \al_I$, as $J' \cap I = \varnothing$. 
Thus the assertion (2) holds.
\end{proof}

\end{document}